\documentclass[12pt]{article}
\usepackage{amsmath, amsthm}
\usepackage{amscd}
\usepackage{amssymb}
\usepackage{array}
\usepackage{soul}
\usepackage{color}

\newtheorem{thm}{Theorem}[section]
\newtheorem{theorem}[thm]{Theorem}

\newtheorem{lemma}[thm]{Lemma}
\newtheorem{proposition}[thm]{Proposition}

\theoremstyle{definition}
\newtheorem{definition}[thm]{Definition}
\newtheorem{example}[thm]{Example}

\newtheorem{observation}[thm]{Observation}

\newtheorem{remark}[thm]{Remark}

\begin{document}

\newcommand{\comment}[1]{{\color{blue}\rule[-0.5ex]{2pt}{2.5ex}}
\marginpar{\scriptsize\begin{flushleft}\color{blue}#1\end{flushleft}}}

\newcommand{\be}{\begin{equation}}
\newcommand{\ee}{\end{equation}}
\newcommand{\bea}{\begin{eqnarray}}
\newcommand{\eea}{\end{eqnarray}}
\newcommand{\bean}{\begin{eqnarray*}}
\newcommand{\eean}{\end{eqnarray*}}

\newcommand{\kk}{{k}}
\newcommand{\R}{\mathbb{R}}
\newcommand{\C}{\mathbb{C}}
\newcommand{\Z}{\mathbb{Z}}
\newcommand{\N}{\mathbb{N}}
\newcommand{\bM}{\mathbb{M}}
\newcommand{\g}{\mathfrak{G}}
\newcommand{\epsi}{\varepsilon}

\newcommand{\hs}{\hfill\square}
\newcommand{\hbs}{\hfill\blacksquare}

\newcommand{\bp}{\mathbf{p}}
\newcommand{\bmax}{\mathbf{m}}
\newcommand{\bT}{\mathbf{T}}
\newcommand{\bU}{\mathbf{U}}
\newcommand{\bP}{\mathbf{P}}
\newcommand{\bA}{\mathbf{A}}
\newcommand{\bm}{\mathbf{m}}
\newcommand{\bIP}{\mathbf{I_P}}

\newcommand{\cA}{\mathcal{A}}
\newcommand{\cB}{\mathcal{B}}
\newcommand{\cN}{\mathcal{N}}
\newcommand{\cC}{\mathcal{C}}
\newcommand{\cI}{\mathcal{I}}
\newcommand{\Ogh}{\tilde{\mathcal{O}}(G/P)}
\newcommand{\cM}{\mathcal{M}}
\newcommand{\cO}{\mathcal{O}}
\newcommand{\cG}{\mathcal{G}}
\newcommand{\cU}{\mathcal{U}}
\newcommand{\cJ}{\mathcal{J}}
\newcommand{\cF}{\mathcal{F}}
\newcommand{\cV}{\mathcal{V}}
\newcommand{\cP}{\mathcal{P}}
\newcommand{\ep}{\mathcal{E}}
\newcommand{\E}{\mathcal{E}}
\newcommand{\cH}{\mathcal{O}}
\newcommand{\cPO}{\mathcal{PO}}
\newcommand{\cl}{\ell}
\newcommand{\cFG}{\mathcal{F}_{\mathrm{G}}}
\newcommand{\cHG}{\mathcal{H}_{\mathrm{G}}}
\newcommand{\Gal}{G_{\mathrm{al}}}
\newcommand{\cQ}{G_{\mathcal{Q}}}

\newcommand{\ri}{\mathrm{i}}
\newcommand{\re}{\mathrm{e}}
\newcommand{\rd}{\mathrm{d}}

\newcommand{\rGL}{\mathrm{GL}}
\newcommand{\rSU}{\mathrm{SU}}
\newcommand{\rSL}{\mathrm{SL}}
\newcommand{\rSO}{\mathrm{SO}}
\newcommand{\rOSp}{\mathrm{OSp}}
\newcommand{\rSpin}{\mathrm{Spin}}
\newcommand{\rsl}{\mathrm{sl}}
\newcommand{\rM}{\mathrm{M}}
\newcommand{\rdiag}{\mathrm{diag}}
\newcommand{\rP}{\mathrm{P}}
\newcommand{\rdeg}{\mathrm{deg}}

\newcommand{\M}{\mathrm{M}}
\newcommand{\End}{\mathrm{End}}
\newcommand{\Hom}{\mathrm{Hom}}
\newcommand{\diag}{\mathrm{diag}}
\newcommand{\rspan}{\mathrm{span}}
\newcommand{\rank}{\mathrm{rank}}
\newcommand{\Gr}{\mathrm{Gr}}
\newcommand{\ber}{\mathrm{Ber}}
\newcommand{\coinv}{\mathrm{co}}

\newcommand{\fsl}{\mathfrak{sl}}
\newcommand{\fg}{\mathfrak{g}}
\newcommand{\ff}{\mathfrak{f}}
\newcommand{\fgl}{\mathfrak{gl}}
\newcommand{\fosp}{\mathfrak{osp}}
\newcommand{\fm}{\mathfrak{m}}

\newcommand{\str}{\mathrm{str}}
\newcommand{\Sym}{\mathrm{Sym}}
\newcommand{\tr}{\mathrm{tr}}
\newcommand{\defi}{\mathrm{def}}
\newcommand{\Ber}{\mathrm{Ber}}
\newcommand{\spec}{\mathrm{Spec}}
\newcommand{\sschemes}{\mathrm{(sschemes)}}
\newcommand{\sschemeaff}{\mathrm{ {( {sschemes}_{\mathrm{aff}} )} }}
\newcommand{\rings}{\mathrm{(rings)}}
\newcommand{\Top}{\mathrm{Top}}
\newcommand{\sarf}{ \mathrm{ {( {salg}_{rf} )} }}
\newcommand{\arf}{\mathrm{ {( {alg}_{rf} )} }}
\newcommand{\odd}{\mathrm{odd}}
\newcommand{\alg}{\mathrm{(alg)}}
\newcommand{\sa}{\mathrm{(salg)}}
\newcommand{\sets}{\mathrm{(sets)}}
\newcommand{\SA}{\mathrm{(salg)}}
\newcommand{\salg}{\mathrm{(salg)}}
\newcommand{\varaff}{ \mathrm{ {( {var}_{\mathrm{aff}} )} } }
\newcommand{\svaraff}{\mathrm{ {( {svar}_{\mathrm{aff}} )}  }}
\newcommand{\ad}{\mathrm{ad}}
\newcommand{\Ad}{\mathrm{Ad}}
\newcommand{\pol}{\mathrm{Pol}}
\newcommand{\Lie}{\mathrm{Lie}}
\newcommand{\Proj}{\mathrm{Proj}}
\newcommand{\rGr}{\mathrm{Gr}}
\newcommand{\rFl}{\mathrm{Fl}}
\newcommand{\rPol}{\mathrm{Pol}}
\newcommand{\rdef}{\mathrm{def}}

\newcommand{\uspec}{\underline{\mathrm{Spec}}}
\newcommand{\uproj}{\mathrm{\underline{Proj}}}

\newcommand{\sym}{\cong}
\newcommand{\al}{\alpha}
\newcommand{\lam}{\lambda}
\newcommand{\de}{\delta}
\newcommand{\dd}{\delta}
\newcommand{\D}{\Delta}
\newcommand{\s}{\sigma}
\newcommand{\sig}{\sigma}
\newcommand{\lra}{\longrightarrow}
\newcommand{\ga}{\gamma}
\newcommand{\ra}{\rightarrow}
\newcommand{\tE}{\widetilde E}

\newcommand{\NOTE}{\bigskip\hrule\medskip}

\newcommand{\beq}{\begin{equation}}
\newcommand{\eeq}{\end{equation}}
\newcommand{\ddet}{\mathrm{det}}
\newcommand{\proj}{\mathrm{proj}}
\newcommand{\pr}{\mathrm{pr}}
\newcommand{\Ubar}{\overline{U}}
\newcommand{\cT}{\mathcal{T}}
\newcommand{\cL}{\mathcal{L}}
\newcommand{\cE}{\mathcal{E}}

\newcommand{\zero}[1]{{#1}_{\scriptscriptstyle{(0)}}}
\newcommand{\one}[1]{{#1}_{\scriptscriptstyle{(1)}}}
\newcommand{\two}[1]{{#1}_{\scriptscriptstyle{(2)}}}
\newcommand{\onef}[1]{{#1}_{\scriptscriptstyle{(1_\F)}}}
\newcommand{\twof}[1]{{#1}_{\scriptscriptstyle{(2_\F)}}}
\newcommand{\three}[1]{{#1}_{\scriptscriptstyle{(3)}}}
\newcommand{\four}[1]{{#1}_{\scriptscriptstyle{(4)}}}
\newcommand{\five}[1]{{#1}_{\scriptscriptstyle{(5)}}}

\newcommand{\zeroc}[1]{{#1}_{\scriptscriptstyle{(\tilde 0)}}}
\newcommand{\onec}[1]{{#1}_{\scriptscriptstyle{(\tilde 1)}}}
\newcommand{\twoc}[1]{{#1}_{\scriptscriptstyle{(\tilde 2)}}}
\newcommand{\threec}[1]{{#1}_{\scriptscriptstyle{(\tilde 3)}}}
\newcommand{\fourc}[1]{{#1}_{\scriptscriptstyle{(\tilde 4)}}}
\newcommand{\fivec}[1]{{#1}_{\scriptscriptstyle{(\tilde 5)}}}

\newcommand{\mzero}[1]{{#1}_{\scriptscriptstyle{(0)}}}
\newcommand{\mone}[1]{{#1}_{\scriptscriptstyle{(-1)}}}
\newcommand{\mtwo}[1]{{#1}_{\scriptscriptstyle{(-2)}}}
\newcommand{\mthree}[1]{{#1}_{\scriptscriptstyle{(-3)}}}
\newcommand{\mfour}[1]{{#1}_{\scriptscriptstyle{(-4)}}}

\newcommand{\ot}{\otimes}
\renewcommand{\cot}{\gamma}
\newcommand{\co}[2]{\cot\left({#1}\ot{#2}\right)}

\medskip

\centerline{\huge{\bf Quantum
principal bundles} }

\medskip
\centerline{\huge{\bf on projective bases}}

\medskip
\medskip
\medskip
\medskip
\centerline{\Large{{P. Aschieri$,^{\!1,2,3}$ R. Fioresi$,^{\!4}$ E. Latini$^{4}$} }}

\vskip 0.5 cm
\centerline{$^1${\sl Dipartimento di Scienze e Innovazione 
Tecnologica, Universit\`a del Piemonte Orientale}}

\centerline{{\sl  
Viale T. Michel - 15121, Alessandria, Italy}}

\vskip 0.1 cm

\centerline{$^2${\sl Istituto Nazionale di Fisica Nucleare, Sezione di Torino,
  via P. Giuria 1, 10125 Torino}}

\vskip 0.1 cm
\centerline{$^3${\sl Arnold-Regge Center, Torino, via P. Giuria 1, 10125, Torino, Italy}}
\centerline{\texttt{paolo.aschieri@uniupo.it}}

\vskip 0.5 cm

\centerline{
$^4${\sl Dipartimento di Matematica, Universit\`{a} di
Bologna}}

\centerline{\sl Piazza di Porta S. Donato 5, I-40126 Bologna, Italy}

\centerline{\texttt{rita.fioresi@UniBo.it}, \texttt{emanuele.latini@UniBo.it}}
\bigskip

\bigskip

\begin{abstract}
 The purpose of this paper is to propose a sheaf theoretic approach to 
the theory of quantum principal bundles over non affine bases.
We study noncommutative  principal
bundles corresponding to $G \to G/P$, where $G$ is semisimple group and $P$
a parabolic subgroup.
\end{abstract}

\tableofcontents
\section{Introduction}

A quantum principal bundle is usually described as an algebra
extension $B\subset A$, with $A$ the ``total space'' algebra on which
coacts a quantum group, and $B$ the
``base space'' subalgebra of coinvariant elements. 
Local triviality is encompassed in the notion of  {locally cleft extension}.

In the commutative setting, 
this picture proves to be extremely effective when the base space $M$
is {\sl affine}, that is, when the algebra $B$ is containing all of
the information to reconstruct the base space. For a projective
base, however, the coinvariant ring $B$ consists of just the
constants, so it is not the object of interest anymore.

In this paper we take a very general point of
view on the definition of
{\sl quantum principal bundle} (see Definition \ref{hg-def}), so that we can
accomodate the affine setting mentioned above, but also
the case of projective base, together with a preferred projective
embedding. In our definition a quantum principal bundle is a locally
cleft sheaf of $H$ comodule algebras 
for a given Hopf algebra $H$. In the commutative setting, when the base is affine 
the algebra of global sections (regular functions on the total space)
is an Hopf-Galois extension;  when the base is a projective variety
our notion still makes sense and it actually gives the correct point
of view to proceed to the quantization.

The definition is tested on an important
special case, that when $M$ is the quotient of a 
semisimple group $G$ and a parabolic subgroup $P$. In this case, in fact,
$M=G/P$ is projective, and we can effectively substitute the coinvariant
ring $B$ with
the homogeneous coordinate ring  $\tilde{\cO}(G/P)$ of $G/P$ with respect to a chosen
projective embedding, corresponding to a line bundle $\cL$. 
The line bundle $\cal L$ can be recovered more algebraically via a
character $\chi$ of $P$; the corresponding sections are the {\sl
  semi-coinvariant} elements of  $\cO(G)$  with respect to $\chi$
and generate the homogeneous coordinate ring $\tilde{\cO}(G/P)$ of $G/P$.
In this case the locally cleft sheaf of $H=\cO(P)$-comodule algebras,
denoted $\cF$,  gives the subsheaf of coinvariants $\cF^{\coinv\,\cO(P)}$ that is the
structure sheaf $\cO_{G/P}$ of $G/P$. The relation between this latter
and the homogeneous coordinate ring $\tilde{\cO}(G/P)$ is then as
usual by considering projective localizations (zero degree
subalgebras of the localizations) of  $\tilde{\cO}(G/P)$.

Similarly, in the quantum case, as in \cite{cfg, fg1} we obtain
the quantum homogeneous coordinate ring $\tilde{\cO}_q(G/P)$ as the $\cO_q(P)$-semi-coinvariant elements of the quantum group $\cO_q(G)$, the quantization of the semisimple
group $G$.
Assuming Ore conditions for localizations, we then proceed to obtain
from  $\tilde{\cO}_q(G/P)$ and $\cO_q(G)$ a suitable sheaf $\cF$ of $\cO_q(P)$-comodule
algebras, which will be the quantum principal bundle over the quantum
space obtained through $\tilde{\cO}_q(G/P)$. More
explicitly, the coinvariant subsheaf $\cF^{\coinv\,\cO_q(P)}$ will be
the quantum structure sheaf associated with the (noncommutative)
projective localizations of  $\tilde{\cO}_q(G/P)$.

The quantization of the flag variety $G/P$ and its noncommutative
geomety has recently attracted a lot of attention. The
theory, also following the remarkable
classification of differential calculi over irreducible quantum flag manifolds
in \cite{HK1,HK2},  has been conspicuously developed in the past
years, see for example
\cite{DD, KLS, lz,  ColinROB, ROB, dl}. In particular, the study of
quantum projective space as a quantum homogeneous space has proven
fruitful, however, it has mainly concerned  quantum projective space
as the base space of a quantum principal $U(N-1)$-bundle with
quantum $SU(N)$ total space, i.e., a study not in the projective context. Indeed,
despite the progress on quantum
principal bundles \cite{brz1, bh, brz2,  HKMZ}, the projective
setting, describing quantum versions of  principal bundles $G\to G/P$,
with $P$ parabolic, is yet to be fully understood. The aim of this paper is
to provide a key step in this direction,
 together with an appropriate setting for a future differential calculus on such quantizations.
\\

We summarize the main results by explaining the organization of the paper.

In Section 2 we recall basic notions in Hopf-Galois extensions,
including the inspiring sheaf approach of \cite{pflaum, cipa}.  We then present
our sheaf theoretic definition of quantum principal bundle. We also provide
the example of $\rSL_2(\C)/P$ both in the classical and in the quantum
setting. This serves also as motivation and preparation for the
general theory we develop in later sections.

In Section \ref{qproj-sec}
we discuss quantum homogenous projective
varieties, mainly following
\cite[\S 2]{cfg}.
Starting from a quantum section $d\in \cO_q(G)$, quantum version of the lift
to $\cO(G)$ of the character $\chi$ of $P$ defining the line bundle $\cal
L$ giving the projective embedding of $G/P$, we construct  the
homogeneous ring $\tilde{\cO}_q(G/P)$. 

In Section \ref{QPB},
we develop a general theory for quantum
principal bundles on homogeneous projective varieties. We construct the sheaf $\cF$ of
$\cO_q(P)$-comodule algebras on the
 quantum projective variety $\tilde{\cO}_q(G/P)$ by local data, that is
by considering suitable projective localizations of
$\tilde{\cO}_q(G/P)$, obtained via a corresponding quantum section $d\in\cO_q(G)$. 
As shown in Theorem \ref{main}, if this sheaf is locally cleft we have 
a quantum principal bundle.

In Section \ref{ex-sec}, we exemplify the construction of Section
\ref{QPB} in  the case of quantum projective space.
We prove that quantum projective space
is the base space of a canonical quantum principal bundle with total
space $\cO_q(\rSL_n)$ and structure group $\cO_q(P)$ (quantum parabolic
subgroup of  $\cO_q(\rSL_n)$). 

In Section \ref{TW}, we apply and further develop the results in
\cite{abps} and show that 2-cocycle deformations (twists) of quantum
principal bundles give new quantum principal bundles. 
We construct three classes of quantum principal bundles
on quantum projective spaces.
The first two are locally cleft but not locally trivial. The
total spaces are not Hopf algebras hence they are not
quantum principal bundles on quantum {\sl homogenous} projective space as in the construction presented in Theorem \ref{main}. 
The second and third class are on multiparametric quantum projective
space, the third class being also an example of the construction in
Theorem \ref{main}, with total space the multiparametric special linear
quantum group.
\bigskip

{\bf Acknoledgements}. The authors  wish to thank Prof. T.  Brzezinski,
Prof. F. Gavarini, Prof. T. Lenagan,  Prof. Z. Skoda and  Dr. C. Pagani
for helpful comments. P.A. wishes to thank the Dipartimento di
Mathematica, Universit\`a di  Bologna for the hospitality during 
the collaboration. 
R.F. and E.L. wish to thank the Dipartimento di Scienze e Innovazione
Tecnologica, Universit\`a del Piemonte Orientale, Alessandria, for the 
hospitality during the collaboration. The work of P.A. is partially
supported by INFN, CSN4, Iniziativa Specifica GSS, and by Universit\`a del
Piemonte Orientale. P.A. is also
affiliated to INdAM, GNFM (Istituto Nazionale di Alta Matematica,
Gruppo Nazionale di Fisica Matematica).

\section{Quantum Principal bundles} 
\label{princ-sec}

In the category of locally compact Hausdorff topological spaces, a 
principal bundle is a bundle $E \to M$, with compatibility
requirements regarding the $P$-space structure, for a given
topological group $P$. 
These requirements can be effectively summarized by asking that the map 
$$
E \times  P \longrightarrow E \times_M E \qquad 
(e,p) \mapsto (e,ep)
$$ 
is a homeomorphism, with  $M=E/P$.

We can dualize this picture by replacing spaces with their
function algebras, that is we replace
$E$ with $A=C(E)$, $M$ with
$B=C(M)$ and $P$ with $H=C(P)$. The notion of principal bundle
is then replaced by that of faithfully flat Hopf-Galois extension. The
Hopf-Galois property is the freeness of the $P$-action, and amounts
to the requirement that the pull-back of the above map,  
called canonical map, 
\beq\label{can-map-def}
\chi: A\otimes_B A\to A\otimes H 
\eeq
is a bijection. The faithfully flat
property, or equivalently, the equivariant projectivity conditions
correspond to the principality of the action (see e.g. \cite{brz2}).

In the affine algebraic category we can proceed and give
the same definitions, where in place of $C(E)$, $C(M)$ and
$C(P)$ we take the coordinate
rings of $E$, $M$ and $P$. In fact,
the contravariant functor associating to affine varieties 
their coordinate ring is an equivalence of categories
(see \cite[Proposition 2.6, \S II]{ha} for more details).

However, when we turn to examine the case of projective
varieties, since the above mentioned equivalence of categories does
not hold anymore as stated, but becomes
more involved, we need to take a different approach to the
theory of principal bundles, introducing
the sheaves of functions on our geometric objects.
As it turns out, this approach, despite
its apparent complication and abstraction is very suitable for
quantization.

\subsection{The Classical description}

We start with a description of the classical setting.
  
\begin{definition}\label{p-princ}
Let $E$ and $M$ be topological spaces, $P$ a topological group and 
$\wp: E \lra M$ a continuous function. We say 
that $(E, M, \wp, P)$ is a \textit{$P$-principal bundle} 
(or {\it principal bundle} for short) with total space
$E$ and base $M$, if the following conditions hold
\begin{enumerate}
\item $\wp$ is surjective.
\item $P$ acts freely from the right on $E$.
\item $P$ acts transitively on the fiber $\wp^{-1}(m)$
of each point $m \in M$.
\item $E$ is locally trivial over $M$, i.e. there is an open covering 
$M=\cup U_i$ and homeomorphisms 
$\sigma_i:\wp^{-1}(U_i)  \lra  U_i \times P$
that are $P$-equivariant i.e., 
$\sigma_i(up)=\sigma(m)p$, $p \in P$.
\end{enumerate}
\end{definition}

We can speak of \textit{algebraic, analytic or smooth} $P$-principal
bundles, we just take the objects and the morphism
of Def. \ref{p-princ} in the appropriate
categories. Notice that $\wp$ is open.

In \cite{pflaum} Pflaum gives a sheaf theoretic characterization of
principal bundles, in the category of locally
compact topological spaces, which is very suitable for 
noncommutative geometry.

In the algebraic category, over a field $\kk$,
we can give another characterization of principal bundles,
closely related to Pflaum's one.
For the basic definitions regarding algebraic
groups we refer e.g. to \cite[\S II]{bo}, for Hopf algebras e.g. to \cite{mo},
\cite[Part VII \S 5 ]{BNCG}.

\begin{proposition} \label{pflaum-prop}
Let $\wp:E \lra M$ be a surjective morphism
of algebraic varieties, and $\cO_E$, 
$\cO_M$ the structural sheaves of $E$ and $M$ respectively.  
Let $\cF$ be the sheaf on $M$ defined by $\cF(U)=\cO_E(\wp^{-1}(U))$.
Let $P$ be an affine algebraic group, $H$ the associated Hopf algebra.
Then $E \lra M$ is a principal
bundle if and only if 
\begin{itemize}
\item $\cF$ is a sheaf of $H$ comodule algebras:
for each open $U \subset  M$ , $\cF (U )$ is a right $H$-comodule
algebra and for each open $W \subset U$ the restriction map $r_{UW} :
\cF (U) \to \cF (W)$ is a morphism of  $H$-comodule algebras;

\item There exists an open covering 
$\{U_i\}$ of $M$ such that we have the following algebra isomorphisms
\begin{enumerate}
\item $\cF(U_i)^{\coinv H}\simeq\cO_M(U_i)$ 
\item 
$\cF (U_i) \simeq \cF (U_i )^{\coinv H} \otimes H$,   as left
$\cF(U_i)^{\coinv H}$-modules and right $H$-comodules for all $i$,
\end{enumerate}
where $\cF(U_i)^{\coinv H}:=\{f \in \cF(U_i)\,|\, \delta(f) = f
\otimes 1\}\subset \cF(U_i) $ is the subalgebra of $H$-coinvariant
elements, with
$\delta: \cF(U_i) \to \cF(U_i) \otimes H$ the $H$-coaction.
\end{itemize}
\end{proposition}

We notice that condition (1) establishes $M\simeq E/P$; 
we will identify $M$ and $E/P$, so that
correspondingly $\cF(U_i)^{\coinv H}=\cO_M(U_i)$.  Condition (2)  gives
the local triviality, the transitive action of $P$ on the fiber and the freeness
of the $P$ action on $E$. 
We leave the details of this
characterization to the reader, it will be a small
variation of the argument given in \cite{pflaum}.

\subsection{The Quantum description}

We now proceed and extend
this point of view in order to give the definition of quantum principal
bundle: it is based on \cite{pflaum}
(see also Proposition \ref{pflaum-prop}) and also on \cite{cipa}, but it
is more general since it encompasses the possibility for
the base manifold to be projective.
Furthermore, we take our category to be algebraic.

\medskip
We will work with algebras (not necessarily commutative)
over a field $\kk$ of characteristic $0$, or the ring of Laurent polynomials
$\kk_q=\kk[q,q^{-1}]$, $q$ an indeterminate.
All algebras will be unital and morphisms preserve the unit.
In particular we will work with $H$-comodule algebras $(A, \de)$,
where $\delta$ denotes the Hopf algebra coaction (frequently omitted). Hopf algebras will be with bijective antipode.

\begin{definition} \label{hg-def}
Let $(H, \Delta, \epsi, S)$ be a Hopf algebra and 
$A$ be an $H$-comodule algebra with coaction
$\delta : A \lra A \otimes H$.
Let 
\beq \label{coinvariant-def}
B:=A^{\coinv H} := \{a \in A | \delta(a) = a \otimes 1\}~.
\eeq
The  extension $A$ of the algebra $B$ is called
$H$-\textit{Hopf-Galois} (or simply \textit{Hopf-Galois}) if the 
map
$$
\chi:A \otimes_B A \lra A \otimes H, \qquad \chi=(m_A \otimes id)(id
\otimes_B \delta)
$$
(called the canonical map) is bijective. 
\end{definition}

If $E \lra M$ is a $P$-principal bundle and $E$, $M$ and $P$
are affine algebraic varieties or  differentiable manifolds, then the
algebra of functions (algebraic or differential) on $E$ and $P$ correspond
respectively to the algebras $A$ and $H$ satisfying  Definition
\ref{hg-def}. The algebra $B$ is the algebra of functions on
the base manifold $M$ (see e.g. \cite{brz2}, \cite{abps} for details).

\begin{example} 
\label{hg-trivial}
Let $B$ be an algebra with trivial right $H$-coaction, i.e.,
$\delta (b) = b\otimes 1$ for all $b\in B$.
Consider $H$ as an $H$-comodule algebra with the coaction given by the
 coproduct $\Delta$.  Then $A:=B\otimes H$ is a right $H$-comodule algebra (with the usual tensor product algebra
and right $H$-comodule structure). We have $A^{\coinv \,H}\simeq B$ and 
$\chi : (B\otimes H)\otimes_B (B\otimes H)\simeq B\otimes H\otimes
H \to B\otimes H\otimes {H}\,,~b\otimes h\otimes h^\prime 
\mapsto b\otimes h{h_1^{\prime}}\otimes
{h_2^\prime}$ is easily seen to be invertible; hence
$B\subset A=B\otimes H$ is an $H$-Hopf-Galois extension. 
\end{example}

We denote as usual by $\ell\ast j$ the \textit{convolution product} of two
linear maps
$j: H\to A$, $\ell: H\to A$. It is defined by $\ell\ast
j(h)=\ell(h_1)j(h_2)$ for all $h\in H$. A linear map $j: H\to A$ is
\textit{convolution invertible} if  it exists $j^{-1}:
H\to A$ such that $j^{-1}\ast j=j\ast j^{-1}: H\to A\,,~h\mapsto \epsi(h)1_A$.
If $A$ is a right  $H$-comodule we can require $j: H\to A$ to be a
right $H$-comodule map where $H$ has $H$-comodule structure given by
$\Delta$, i.e.,  $\delta \circ j = ( j \otimes id) \circ
\Delta$.

\begin{definition}\label{cleft-def}
Let $H$ be a Hopf algebra and $A$ an $H$-comodule algebra.
The algebra extension $A^{\coinv\, H}\subset A$ is called
a \textit{cleft extension} if
there is a right $H$-comodule map $j : H \to A$, called \textit{cleaving map}, 
that is convolution invertible.

An 
extension $A^{\coinv\, H}\subset A$ is called a
\textit{trivial extension} if 
there is an $H$-comodule algebra map $j:H\to A$.  
\end{definition}

Since $1_H$ is a grouplike element $j(1_H )  j^{-1}(1_H ) = 1_A$,
so that  $j(1_H)$ is an invertible element in $A^{\coinv\, H}$. Hence a
cleaving map can always be normalised to $j(1_H)=1_A$. We will always consider normalized cleaving maps.

\begin{remark}\label{equivariantly-projective} 
Cleft extensions, if the base ring is a field $\kk$,  are furthermore faithfully flat (or equivariantly
projective) Hopf-Galois
extensions (see e.g. \cite[Part~VII \S 6]{BNCG}, \cite{brz2}).
\end{remark}

\begin{remark}\label{cleft-rem} 
A trivial extension $A^{\coinv\,
H}\subset A$ is automatically a cleft extension.
In fact, since an $H$-comodule algebra map $j:H\to A$  maps the unit of $H$ in
that of $A$, its convolution inverse is
$j^{-1}=j\circ S$. 
Furthermore, the  $H$-comodule algebra map $j:H\to A$ is an injection, indeed the map 
($\varepsilon\otimes id)\circ (m\otimes id)\circ (id\otimes j\!\circ\!
S\otimes id)\circ(\delta\otimes id)\circ \delta $ sends $j(h)$ to
$h$. Thus the subalgebra $j(H)\subset A$ is isomorphic to $H$.

The  extension $B\subset B\otimes H$ of Example \ref{hg-trivial} is an example
of trivial extension (with $j(h)=1_B\otimes h$, for all
$h\in H$).
\end{remark}

By a theorem of Doi and Takeuchi \cite{DT} 
(we also refer to \cite[Theorem 8.2.4]{mo}, \cite[Part~VII \S 5]{BNCG}) cleft extensions are special cases of
Hopf-Galois extensions. 
\begin{theorem}\label{thmDT}
Let $A$ be an $H$-comodule algebra (with base ring a field $\kk$), then $A^{\coinv\, H}\subset A$ is a cleft
extension if
and only if  $A^{\coinv\, H}\subset A$ is an Hopf-Galois extension and
there is an
$H$-comodule and left $B=A^{\coinv\, H}$-module isomorphism
$B\otimes H\simeq A$.
\end{theorem}
Here  $B\otimes H$ is an $H$-comodule
with $H$-coaction $id\otimes \Delta$. For later use we recall that the  relation betweeen a cleaving map $j: H\to A$ and the 
left $B=A^{\coinv\, H}$-module and $H$-comodule isomorphism $\theta:
B\otimes H\to A$ is given by $\theta(b\otimes
h)=bj(h)$.
\\

The notion of cleft extension is the noncommutative generalization of
that of trivial principal bundle. 
The next observation sharpens the relation between
trivial Hopf-Galois
extensions, trivial principal bundles and cleft extensions.

\begin{observation} \label{obsjtheta}
If $j:H\to A$ is an $H$-comodule {\sl algebra} map, then 
we have an action of $H$ on $B=A^{\coinv\, H}$ given by 
$h\triangleright b=j(h_1){{\,}}b{{\,}}j^{-1}(h_2)=j(h_1){{\,}}b{{\,}}j(S(h_2))$, for all $h\in H,
b\in B$. We can therefore consider the smashed product
algebra $B\,\sharp\, H$, that is the $H$-comodule $B\otimes H$ with product structure
$(b\otimes h)(b'\otimes h')=b_{\:\!}(h_1\triangleright b')\otimes h_2h'$.
With this product $\theta: B\,\sharp\, H\to A$ is an $H$-comodule algebra
isomorphism. 
If $B$ is central the smashed product is the usual tensor
product of algebras. In particular, in
the affine case, we immediately recover that a
$P$-principal bundle $E\to E/P$ is trivial if and only if
${\cO}(E/P)\otimes \cO(P)\simeq \cO(E)$ as $\cO(P)$-comodule algebras.

In the more general case of an extension that is nontrivial but cleft,
the map $j:H\to A$ is not an $H$-comodule
algebra map, and the 2-cocycle $$\tau: H\otimes H\to
B~, ~~\tau(h,k)=j(h_{(1)}) j(k_{(1)}) j^{-1}(h_{(2)}k_{(2)})$$ measures
this failure.  In general the map $h\otimes b\mapsto j(h_1){{\,}}b{{\,}}j^{-1}(h_2)$
is not an action of $H$ on $B$.
In this cleft case we can still induce via the isomorphism 
$\theta: B\otimes H\to A$ an algebra structure on $B\otimes H$,
this corresponds to a crossed product $B\,\sharp_\tau H$
(see e.g. \cite[Proposition 7.2.3]{mo}). 
\end{observation}
\medskip

\medskip

\medskip

We want to present a notion of quantum principal bundle that is more
general than that of Hopf-Galois extension presented in Def. 
\ref{hg-def}, and which can accomodate also the case where $M$ is an
algebraic variety, which is not affine.
To this end, we  consider a {\sl sheaf theoretic} description
of quantum principal bundles. We start by introducing
the notion of {\sl quantum ringed space}.

\begin{definition}\label{qringed-def}
A \textit{quantum ringed space}  
$(M , \cO_M)$ is
a pair consisting of a classical topological space $M$ and 
a sheaf over $M$ of noncommutative algebras.
\end{definition}

Classical differentiable manifolds 
or algebraic varieties, together with the sheaves of functions on them
(differentiable or algebraic) are examples of quantum ringed spaces.
Also supergeometry provides important examples (see \cite{ccf} Ch. 3).
We now define  the key notion of {\sl quantum principal bundle}
by extending to the quantum case 
what we established in Proposition {\ref{pflaum-prop}}.

\begin{definition}\label{hg-shf} 
Let $(M,\cO_M)$ be a ringed space and $H$ a Hopf algebra. 
We say that a sheaf of $H$-comodule algebras $\cF$ 
is an $H$-\textit{principal bundle} 
or \textit{quantum principal bundle} over $(M,\cO_M)$
if there exists an open covering 
$\{U_i\}$ of $M$ such that:
\begin{enumerate}
\item $\cF(U_i)^{\coinv H}=\cO_M(U_i)$,
\item $\cF$ is \textit{locally cleft}, 
that is 
$\cF (U_i )$ is a cleft extension of $\cF(U_i)^{\coinv H}$.
\end{enumerate}
The locally cleft property is equivalent to
the existence of a \textit{projective cleaving map}
that is a collection of cleaving maps $j_i:H \lra
\cF(U_i)$.
\end{definition}

\begin{remark} 
A sheaf of Hopf-Galois extentions is locally cleft if it so as a sheaf
of $H$-comodule algebras. A locally cleft sheaf $\cF$ of Hopf-Galois extensions is in
particular a quantum principal bundle on the quantum ringed space $(M,\cF^{\coinv H})$.

Moreover, a sheaf $\cF$ of $H$-comodule algebras, such that
the extension $\cO_M(M)=\cF(M)^{\coinv H}\subset \cF(M)$ is Hopf-Galois, is
equivalent to a sheaf of Hopf-Galois extensions, indeed, as observed
in \cite{cipa}, the property of being Hopf-Galois restricts locally. 
Therefore, a quantum principal bundle $\cF$ has the property 
$\cO_M(M)=\cF(M)^{\coinv H}\subset \cF(M)$ is Hopf-Galois, if and only if it is a
locally cleft sheaf of Hopf-Galois extensions.
\end{remark}
Let us see a simple example, in the commutative setting, that we will
generalize to the noncommutative setting and generic dimensions.

\begin{example} \label{sl2-ex}
Let $E=\rSL_2(\C)$ and consider the principal bundle 
$\wp: \rSL_2(\C) \lra \rSL_2(\C)/P \simeq \bP^1(\C)$,
where $P$ is the upper Borel in $\rSL_2(\C)$, i.e.,  the subgroup  of all
 matrices with vanishing entry (1,2).
Let $A=\cO(\rSL_2)$ be the algebra of regular functions on the complex
special linear group $\rSL_2(\C)$. We explicitly have
$$
\cO(\rSL_2)=\C[a,b,c,d]/(ad-bc-1)~,
$$
where $\C[a,b,c,d]$ denotes the commutative algebra over $\C$ freely generated
by the symbols $a,b,c,d$, while $(ad-bc-1)$ denotes the ideal generated by the
element $ad-bc-1$, that implements the determinant relation.

Let  $\cO(P)$ be the algebra of functions on $P\subset \rSL_2(\C)$,
this is the quotient $\cO(\rSL_2)/(c)
=\C[t,p,t^{-1}]:=\C[t,p,s]/(ts-1)$.
With the comultiplication $\Delta$ in $\cO(\rSL_2)$ 
and the projection
$$\pi:\cO(\rSL_2)\lra \cO(\rSL_2)/(c)~~$$ that on the generators reads
$\left({}_c^a~{}_d^b\right)\mapsto \left({}_0^t~{}_{t^{-1}}^p\right)$
(and is extended as an algebra map)
 we can define the coaction 
\begin{equation}\label{deltaSL2}
\de=(id \otimes \pi)\Delta : \cO(\rSL_2)\rightarrow \cO(\rSL_2)\otimes
{\cO}(P)~.
\end{equation}

The coinvariants $B=A^{\coinv\, {\cO(P)}}$ of this coaction are just the constants,
indeed the coinvariant are functions on the base space  $\bP^1(\C)$,
and the only regular functions on all projective space are the
constants (Liouville theorem).
We see that the extension $A^{\coinv\, {\cO(P)}}\subset A$ is not
    Hopf-Galois, and that this is due to the lack of regular functions
    on the base space of the $P$-principal bundle
    $\wp : \rSL_2(\C)\to \rSL_2(\C)/P\simeq \bP^1(\C)$. 

Nevertheless, we can define an $\cO(P)$-principal bundle structure
according to Definition \ref{hg-shf}. 
To this aim, we first consider an affine open cover of the total space 
and then we project it to the base.

Let $\{V_1,V_2\}$ be the open cover of $\rSL_2 (\C)$ where
$V_i$ consists of those matrices in $\rSL_2 (\C)$ with entry $(i,1)$ not equal to zero. Define $U_i=\wp(V_i)$
and observe that
$\{U_1,U_2\}$ is an open cover of $\bP^1(\C)$ since $\wp$ is an open
map.
The algebras of functions on the opens $V_1$ and $V_2$ are the
localizations 
$$A_1:=\cO(\rSL_2)[a^{-1}]=A[a^{-1}]~,~~~A_2:=\cO(\rSL_2)[c^{-1}]=A[c^{-1}]~.$$ 
The coaction in (\ref{deltaSL2}) uniquely extends to coactions $\de_i:A_i \lra A_i \otimes  {\cO}(P)$ on
these localizations (namely $\delta a^{-1}=a^{-1}\otimes t^{-1}$, 
$\delta c^{-1}=c^{-1}\otimes t^{-1}$).
The coinvariant subalgebras  $B_i=A_i^{\coinv \, {\cO}(P)}$ 
explicitly read
$$
B_1=\C[a^{-1}c] \simeq \C[z], \qquad B_2=\C[ac^{-1}] \simeq \C[w]~.
$$
Notice that they are the coordinate rings
of the affine algebraic
varieties $U_i\simeq \C$ open in $\rSL_2(\C)/P\simeq \bP^1(\C)$.

Next we consider on  $\bP^1(\C)$ the topology
$\{\emptyset, U_{12}=U_1\cap U_2, U_1, U_2, \bP^1(\C)\}$ (this is a rough
topology, but sufficient to describe the principal bundle on
$\bP^1(\C)$). We then define the ringed space $(\bP^1(\C),
\cO_{\bP^1(\C)})$ with sheaf of regular functions $\cO_{\bP^1(\C)}$ given by
$$\cO_{\bP^1(\C)}(U_i):=B_i~,~~\cO_{\bP^1(\C)}(U_{12}):=B_{12}:=B_1[z^{-1}]~,~~\cO_{\bP^1(\C)}(\bP^1(\C)):=\C$$
and with $\cO_{\bP^1(\C)}(\emptyset)$ being the one element algebra
over $\C$, terminal object in the category of algebras.
It is easy to verify that the restriction morphism 
$r_{{12},2}:B_2\to B_{12}$, $w\mapsto z^{-1}$,
with all other ones being given by the obvious inclusions (but for the
empty set where we have the canonical projections),
indeed define the sheaf of regular functions on $\bP^1(\C)$.

Finally we define the sheaf $\cF$ of $\cO(P)$-comodule algebras
$$
\cF(U_i):=A_i~, \quad \cF(U_{12}):=A_{12}:=A_1[c^{-1}]=A_2[a^{-1}]~, 
\quad \cF(\bP^1(\C))=\cO(\rSL_2),
$$
and $\cF(\emptyset):=\{0\}$ (the  one element algebra) with the
obvious restriction morphisms.

We now show that all properties required by Def. \ref{hg-shf}
are satisfied. Indeed by construction $\cO(U_i)=B_i=A_i^{\coinv\,
  \cO(P)}=\cF(U_i)^{\coinv \cO(P)}$. Furthermore the $\cO(P)$-comodule
  $\cF(U_1)$ is a trivial extension (and hence a cleft extension)
  because the map $j_1 :\cO(P)\to  A_1$ defined on the generators by 
$$
t^{\pm 1}\mapsto a^{\pm 1}~,~~ p\mapsto b~,
$$
and extended as algebra morphism to all $\cO(P)$ is well defined and
easily seen to be an $\cO(P)$-comodule morphism (recall
$\delta a^{\pm1}=a^{\pm 1}\otimes t^{\pm 1}$ and $\delta b=b\otimes
t^{-1}+a\otimes p$). Similarly, $\cF(U_2)$ is a trivial extension with
$j_2:\cO(P)\to  A_2$ given by  $t^{\pm 1}\to c^{\pm 1} $, $ p\mapsto d$.
\end{example}

\begin{example}\label{sl2q}

We discuss the quantum 
deformation of the previous example. Consider the algebra 
$A_q$ that is the algebra  $\C_q\langle a, b, c, d\rangle$ freely
generated (over $\C_q=\C[q,q^{-1}]$, $q$ an indeterminate that may be
specialized to a complex number) 
by the symbols $a,b,c,d$, modulo the ideal $I_M$
generated by the
$q$-commutation relations (or Manin relations,
cf. Def. \ref{qmatrices}),
$$ 
\begin{array}{c}
ab=q^{-1}ba, \quad ac=q^{-1}ca, \quad bd=q^{-1}db, \quad cd=q^{-1}dc, \\ \\
bc=cb \qquad ad-da=(q^{-1}-q)bc 
\end{array}
$$
and modulo
the ideal $ (ad-q^{-1}bc-1) $ generated by the determinant
relation. In short: 
$$
A_q:=\cO_q(\rSL_2)\,=\,
\C_q\langle a, b, c, d \rangle /I_M + (ad-q^{-1}bc-1)~.
$$
Let us similarly define
$$
\cO_q(P):= 
\C_q \langle t, t^{-1}, p \rangle/(tp-q^{-1}pt):=\C_q \langle t, s, p
\rangle/  (ts-1, st-1, tp-q^{-1}pt) ~.
$$

Let $U_i$ be a cover of $M=\rSL_2(\C)/P$
as in Example \ref{sl2-ex}.
In analogy with the classical case we define $A_{q\,1}:=A_q[a^{-1}]$, $A_{q\,2}:=A_q[c^{-1}]$, the
noncommutative localizations in the elements $a$ and $c$ respectively. The coinvariants are given by
 $$
B_{q\,1}=\C_q[a^{-1}c] \simeq \C_q[u], \qquad B_{q\,2}=\C_q[ c^{-1}a] \simeq \C_q[v]~.
$$
and the ringed space $(\bP^1(\C),
\cO_{q\,\bP^1(\C)})$ can be then easily constructed in analogy with the commutative case:
$$
\cO_{q\,\bP^1(\C)}(U_i):=B_{q\,i}~,~~\cO_{q\,\bP^1(\C)}(U_{12}):=B_{q,12}
:=B_{q,1}[u^{-1}]~,~~\cO_{q\,\bP^1(\C)}(\bP^1(\C)):=\C
$$
with the nontrivial restriction map given by $r_{q\,{12},2}:B_{q\,2}\to B_{q\,12}$, $v\mapsto u^{-1}$ that is again well defined since on $U_{12}$ one has $uv=1=vu$.

 The natural candidate 
$$
\cF(U_i):=A_{q\,i}~, \quad \cF(U_{12}):=A_{q\,12}:=A_{q\,1}[c^{-1}]=A_{q\,2}[a^{-1}]~, \quad \cF(\bP^1(\C))=A_q~,
$$
 is again a sheaf of $\cO_q(P)$-comodule algebras on $\bP^1(\C)$; note in particular that $A_{q\,12}$ is well defined since the localization we choose satisfies the Ore condition (see \cite{sk1}). 
As in the previous section we define the  cleaving maps 
$j_i:\cO_q(P) \lra A_{q,i}$, $i=1,2$ on the generators as:
$$
\begin{array}{rl}
j_1:& t^{\pm 1}\mapsto a^{\pm 1},\qquad p\mapsto b~, \\ \\
j_2:& t^{\pm 1}\to c^{\pm 1}, \qquad p\mapsto d.
\end{array}
$$
We observe that $j_1$ extends to an algebra map to all $A_{q,1}$:
$$
 j_1(tp-q^{-1}pt)=j_1(t)j_1(p)-q^{-1}j_1(p)j_1(t)=ab-q^{-1}ba
$$
and similarly for $j_2$.
The comodule property of $j_1$ (and similarly for $j_2$) is
then easily checked on the generators:
 $$
 \delta \circ j_1(t)=a\otimes t=(j_1\otimes id) \circ\Delta( t) 
 $$
 and
 $$
 \delta \circ j_1(p)=b\otimes t^{-1}+a\otimes p=(j_1\otimes id) \circ\Delta( p)~.
 $$
We can then conclude that $A_{q\,i}$ are trivial $\cO_q(P)$-extensions of $B_{q\,i}$.
\end{example}

We will study a generalization of the above example
in Section \ref{ex-sec}. In that more general setting we will
use the following proposition (see e.g. \cite[\S 1.1]{eh}),
\begin{proposition} \label{sheaf-prop}
~1.~Let $\cB$ be a basis
for a topology $\cT$ on $M$. Then a $\cB$-sheaf
of $H$-comodule algebras $\cF$ (that is a sheaf defined for the open sets in $\cB$ with 
gluing conditions) extends to a unique sheaf of $H$-comodules on $M$.\\[-.5em]

2.~If $\{U_i\}$ is an open cover of $M$, then 
the empty set and  finite intersections 
$U_{i_1} \cap \dots \cap U_{i_r}$ form a basis for a topology on $M$.
\end{proposition}

\begin{remark}In Example \ref{sl2q}, with $\{U_i\}$ open cover of 
$\bP^1(\C)$, the
$\cB$-sheaf is the restriction of $\cF$ to $\cB=\{\emptyset, U_{12}, 
U_1, U_2\}$, and $\cF(\bP^1(\C))$ is recovered as the pull-back $\cF(\bP^1(\C))=\{(f,g)\in \cF(U_1)\times
\cF(U_2)\,;\, r_{q\;12,1}(f)=r_{q\;12,2}(g)\}$  of
$\cO_q(P)$-comodule algebras (here $r_{q\;12,i}:A_{q\: i}\to
A_{q\:12}$ are the obvious restriction maps).
\end{remark}

\section{Quantum homogenous projective varieties}
\label{qproj-sec}
A homogenous projective variety can be realized as quotient of affine
algebraic groups $G$, $P$. Its homogenous coordinate ring
$\tilde{\cO}(G/P)$ with respect to a chosen projective embedding, 
when corresponding to a very ample line bundle ${\cal L}$, is
obtained via a section of ${\cal L}$; this is a a given element
$t\in\cO(G)$.
A quantum homogenous projective variety  $\tilde{\cO}_q(G/P)$  can be
similarly characterized via a quantum section $d\in\cO_q(G)$.
We review this construction due to
\cite{cfg}, see also \cite{fg1}, adapting, for the reader's convenience,
the main definitions and results to the present setting that differs from the
first reference setting (there the accent was on Poisson geometry and Quantum
Duality principle).

\subsection{Projective embeddings of homogeneous spaces}

If $G$ is a semisimple algebraic group, $P$ a parabolic subgroup,
the quotient $G/P$ is a projective variety
and the projection
$G \lra G/P$ is a principal bundle (see Definition \ref{p-princ}). 
$G/P$ is an homogeneous space for the $G$-action and just
an homogeneous variety for the $P$-action, which is not
transitive. 

 We now recall how a character of $P$ determines a projective
 embedding of $G/P$ and its coordinate ring $\Ogh$.
Given a representation  $\rho$ of $P$ on some vector space $V$,  we can
construct a vector bundle associated to it, namely
$$  
\cV := G \times_P V = G \times V/ \simeq \, , \quad
(gp,v) \simeq (g, \rho(p)^{-1}v) \, , \qquad \forall p \in P,
 g \in G, v \in V .  
$$
The space of global sections of this bundle is identified with the induced
module (see, e.g., \cite{ha} for more details)
$$  
H^0 \big( G\big/P, \cV \big)  \, = \,  
\big\{ f \colon G \rightarrow V \,\big|\, f \text{\ is regular},
\, f(gp) = \rho(p)^{-1}f(g) \,\big\} \; .  
$$
In particular, for $ \, \chi \colon \, P \lra \kk^* \, $  a character of  $ P $,
i.e.~a one dimensional representation of  $ P $  on  $ \, L \simeq \kk
\, $, we can consider 
 $ \, \cL^n := G \times_P L^{\otimes n} \, $ and define
  $$  \displaylines{
   \Ogh_n  \; := \; H^0 \big( G\big/P, \cL^{n} \big)  \cr\cr
   \Ogh  \; := \;  {\textstyle \bigoplus_{n \geq 0}}
\, \Ogh_n \; \subset \;  \cO(G)  \quad . }  $$
Assume $ \cL $  is very ample, i.e. it is generated 
by a set of global sections  $ f_0 $,  $ f_1 $,  $ \dots $,  
$ f_N \in \Ogh_1 \:\!$; so that the algebra  $ \Ogh $  
{\sl is graded and generated in degree 1\/}
(by the  $ f_i $'s).  Then  $\Ogh $  is the homogeneous coordinate ring of the
projective variety  $ G\big/P$  with respect to the embedding  given via
the global sections of  $\cL $  (see \cite{gh}, p.~176).

\begin{observation}While $\cO_{G/P}$ denotes the structure sheaf of $G/P$,
    so that $\cO_{G/P}(G/P)$ is the space of global sections, that is
      $\kk$ since $G/P$ is a projective variety, $\Ogh$ denotes
      the homogeneous coordinate ring of $G/P$.
\end{observation}

\smallskip

   We want to reformulate this classical construction in purely Hopf algebraic
terms. 
The character $\chi$  is a group-like element in the coalgebra
$ \cO(P) \, $.  The same holds for all powers  $ \chi^n $  ($ n \! \in \! \N
\, $). 
As the  $ \chi^n $'s  are
group-like, if they are pairwise different they also are linearly independent,
which ensures that the sum  $ \, \sum\limits_{n \in \N} \Ogh_n \, $, inside
$ \cO(G) $, is a direct one.  Moreover, once the embedding is given, each
summand  $ \Ogh_n $  can be described in purely Hopf algebraic terms as
\beq
\label{scoinv}
\begin{array}{rl}
   \Ogh_n  \, &\!\!:= \,  \big\{ f \in \cO(G) \,\big|\, f(gp) =
\chi^n\big(p^{-1}\big) f(g) \,\big\} \,   \\ \\
      &= \, \Big\{\, f \in \cO(G) \,\Big|\, \big((\text{\it id}
\otimes \pi) \circ \Delta \big)(f) = f \otimes S\big(\chi^n\big) \Big\}
\end{array} 
\eeq
 with  $ \, \pi \, \colon \cO(G) \relbar\joinrel\rightarrow \cO(P)
\, $  the standard projection,  $ S $  the antipode of  $ \cO(P) \, $.
Lifting $S(\chi)\in \cO(P)$ to an element $t\in  \cO(G)$ we have
the following proposition.

\begin{proposition} \label{t}
Let $P$ be a parabolic subgroup of a semisimple
algebraic group $G$ and denote by $\pi: \cO(G) \lra \cO(P)$ the natural projection
dual to the inclusion $P \subset G$. If $G/P$ is embedded into some
projective space via some very ample line bundle $\cal L$ then there
exists an element $t \in \cO(G)$ such that
\beq  \label{eq2.2}
\Delta_\pi(t) \, := \, \big((\text{\it id} \otimes \pi)
\circ \Delta \big)(t) \, = \, t \otimes \pi(t)     
\eeq
\beq \label{eq2.3}
\pi\big(t^m\big) \not= \pi\big(t^n\big) \quad \forall \;\; m \not= n 
\in \N   
\eeq
\beq \label{eq2.4}
\Ogh_n  \; = \;  \Big\{\, f \in \cO(G) \,\;\Big|\;
(id \otimes \pi)\Delta(f) = f \otimes \pi\big(t^n\big) \Big\}   
\eeq
\beq\label{eq2.5}
\Ogh  \; = \;  {\textstyle \bigoplus_{n \in \N}} \; \Ogh_n   
\eeq
where  $ \, \Ogh \, $  is the homogeneous coordinate ring generated by
the global sections of $\cal L$, i.e. generated by $ \, \Ogh_1 \, $. 

Vice-versa, given $t \in \cO(G)$ satisfying \eqref{eq2.2},
\eqref{eq2.3}, if $\tilde{\cO}(G/P)$ as defined in \eqref{eq2.4}, \eqref{eq2.5}
is  generated in degree 1, namely by
$ \Ogh_1 \, $, then  $\tilde{\cO}(G/P)$ is the homogeneous
coordinate ring of the projective variety $G/P$ associated with
the projective embedding of $G/P$ given by the very ample line
bundle ${\cal L}=G\times_P\kk$,
the $P$-action on the ground field $\kk$ being induced by $\pi(t)$.
\end{proposition}

\begin{proof} See \cite{cfg}. \end{proof}
  Notice that while  $ \, S(\chi) = \pi(t) \, $  {\sl is\/}  group-like,
$ t $ 
has an ``almost
group-like property'', given by (\ref{eq2.2}). 
We call an element  $t \in \cO(G)$ satisfying \eqref{eq2.2},
\eqref{eq2.3} a \textit{classical section}   because 
$t\in \tilde{\cO}(G/P)_1$. 
The line bundle $\cal L$ and the homogenous coordinate ring
$\, \Ogh \, $ depend only on $\,\pi(t)$, not on the lift $t$.

\medskip

\begin{remark}  \label{rem-Ogh}
 We point out that  $ \Ogh $  is a unital subalgebra  as well
 as
a (left) coideal of  $ \cO(G)$;  the latter property reflects the fact that
$ G\big/P $  is a (left)  $ G $--space.  Thus, the restriction of the
comultiplication of  $ \cO(G) \, $,  namely
  $$  \Delta\big|_{\tilde{\cO}(G/P)} : \Ogh \lra \cO(G) \otimes \Ogh \quad ,  $$
is a left coaction of  $\cO(G)$ on  $ \Ogh $,  which structures  $\Ogh $
into an  $ \cO(G) $--comodule  algebra. Moreover $\Ogh$ is {\sl graded\/}  and the coaction
$ \Delta\big|_{\tilde{\cO}(G/P)} $  is also {\sl graded\/}  with
respect to the
trivial grading on  $ \cO(G) \,$,  so that each  $ \Ogh_n $  is
indeed a coideal of  $ \cO(G) $ as well.
\end{remark}

\subsection{Quantum homogeneous projective varieties and    quantum sections}
We quickly recall some definitions of quantum deformations and
quantum groups, establishing our notation. We define quantum homogeneous spaces and then
turn to the quantization of the picture described in the previous
section.

\begin{definition}  \label{qgrp}
By  {\it quantization\/}  of  $\cO(G)$,  we mean a Hopf algebra
$\cO_q(G)$  over the ground ring  $\kk_q := \kk[q,q^{-1}]$, 
where $q$ is an indeterminate, such that:
  
\begin{enumerate}
\item the  specialization  of  $ \cO_q(G) $  at
$ \, q = 1 \, $,  that is  $ \, \cO_q(G) / (q\!-\!1) \, \cO_q(G)
\, $,  is isomorphic to  $ \cO(G) $  as an Hopf algebra;

\item
$ \cO_q(G) $  is torsion-free, as a  $\kk_q$--module;
  
 \end{enumerate}
We also call  $ \cO_q(G) $  
a  {\it quantum deformation} of  $ G \, $,  or for short,
{\it quantum group}.

We also say that the $\kk_q$-algebra $ \cO_q(M) $  is a
{\it quantization\/}  of  $ \cO(M) $
if it is torsion-free and  $ \, \cO_q(M) / (q-1) \cO_q(M)
\simeq \cO(M) \, $.  If $\cO(M)$ is the coordinate ring of an
affine variety $M$, we further say that $\cO_q(M)$ is a quantization
of $M$.
If $\tilde{\cO}(M)$ is the homogeneous coordinate ring of a projective variety,
with respect to a given projective embedding, we say that
 $\tilde{\cO}_q(M)$ is a quantization
of $M$ provided it is graded and the quantization preserves the
homogeneous components.
\end{definition}

We next define quantum homogeneous varieties, in this case $M=G/P$.

\begin{definition} \label{qhv-def}
Let $G/P$ be a homogeneous space with respect
to the action of an algebraic group $G$. 
If $G/P$ is affine we say that its quantization
$\cO_q(G/P)$ is a \textit{quantum homogeneous variety (space)} if  
$\cO_q(G/P)$ is a subalgebra of $\cO_q(G)$ and an
$\cO_q(G)$-comodule algebra. 
If $G/P$ is projective and $\tilde{\cO}(G/P)$ is its homogeneous coordinate
ring with respect to a given projective embedding, then we ask its
quantization $\tilde{\cO}_q(G/P)$ to be a  $\cO_q(G)$-comodule
subalgebra of $\cO_q(G)$.
We furtherly ask the algebra $\tilde{\cO}_q(G/P)$ to be graded and the
$\cO_q(G)$-coaction to preserve the grading.
In this case
we call $\tilde{\cO}_q(G/P)$ a \textit{quantum homogeneous projective
  variety}.
\end{definition}

Let $\cO_q(G)$ be a quantum group and $\cO_q(P)$ a quantum
subgroup (quotient Hopf algebra),  quantizations respectively of $G$ and $P$ as above.
Since from Proposition \ref{t} a classical section $t$ defines a
line bundle on $G/P$ and a projective embedding, we study a quantum
projective embedding by  quantizing this classical section.

\begin{definition} \label{qsec}
A  {\it
quantum section\/}  of the line bundle  $ \cL \, $  on  $ G \big/ P $
associated with  the classical
section $t$, is an element  $ \, d \in \cO_q(G) \, $  such that

\begin{enumerate}
\item $(id \otimes \pi) \Delta(d) = d \otimes \pi(d) \; $,  i.e.
$ \Delta(d) -d \otimes d \,\in\, \cO_q(G) \otimes I_q(P) $

\item $ d \equiv t$, mod$(q-1)$
\end{enumerate}
where $\pi: \cO_q(G) \lra \cO_q(P):=\cO_q(G)/I_q(P)$, $I_q(P)\subset \cO_q(G)$ being a 
Hopf ideal, quantization of the Hopf ideal $I(P)$ defining $P$.
\end{definition}

Define now:
\beq \label{G/P-eq}
\begin{array}{rl}
\tilde{\cO}_q(G/P)&:=\sum_{n\in \mathbb{N}} \tilde{\cO}_q(G/P)_n, \quad \hbox{where} \\ \\
\tilde{\cO}_q(G/P)_n &:= \{f \in \cO_q(G)\, | \, 
(\text{\it id} \otimes \pi)\Delta(f) = f \otimes \pi\big(d^n\big)\}.
\end{array}
\eeq
We recall a result from \cite{cfg}.

\begin{theorem}  \label{Oqgh-graded}

   Let  $ d $  be a quantum section on  $ \, G\big/P \, $.  Then

\begin{enumerate}
\item $ \tilde{\cO}_q(G/P) \, $  is a  graded algebra, 
$$
\tilde{\cO}_q(G/P)_r \cdot \tilde{\cO}_q(G/P)_s \subset \tilde{\cO}_q(G/P)_{r+s}, \quad
\tilde{\cO}_q(G/P) ={\bigoplus}_{n \in \N} \tilde{\cO}_q(G/P)_n ~.
$$
 \item 
$ \,\tilde{\cO}_q(G/P) $  is a  {\sl graded}
$ \, \tilde{\cO}_q(G) $--comodule  algebra, via
the restriction of the comultiplication $\Delta$
in $\cO_q(G)$, 
$$
\Delta|_{\tilde{\cO}_q(G/P)}: \tilde{\cO}_q(G/P) \lra \cO_q(G) \otimes 
\tilde{\cO}_q(G/P)~
$$
where we consider $\cO_q(G)$ with the trivial grading.
\item As algebra $\tilde{\cO}_q(G/P) \, $  is 
  a subalgebra of $\cO_q(G)$. 
\end{enumerate}
Hence  $\tilde{\cO}_q(G/P)$ is a quantum 
homogeneous projective variety.                                
\end{theorem}

From now on we assume that $\tilde{\cO}_q(G/P)$ is generated in degree
one, namely by $\tilde{\cO}_q(G/P)_1$.
The quantum Grassmannian and flag are examples of this construction
and they are both generated in degree one.

\begin{example}\label{qgrass-ex}
Let us consider the
case $G=\rSL_n(\C)$ and $P$ the maximal parabolic subgroup of $G$:
$$
P=\left\{ \begin{pmatrix} t_{r \times r} & p_{r \times n-r} \\
{ 0}_{n-r \times r} & s_{n-r \times n-r} \end{pmatrix}
\right\} \, \subset  \, \rSL_n(\C)~.
$$ 
The quotient
$G/P$ is the Grassmannian $\Gr$ of $r$ spaces into the
$n$ dimensional vector space $\C^n$. It is a projective variety
and it can be embedded, via the Pl\"ucker embedding, into
the projective space $\bP^{N}(\C)$ where $N=\begin{pmatrix} n \\ r\end{pmatrix}$.
This embedding corresponds to the character:
$$
P \ni \begin{pmatrix} t & p \\
0 & s \end{pmatrix} \mapsto \det(t) \in \C^\times~.
$$
The coordinate ring $\cO(\Gr)$ of $\Gr$, with respect to the Pl\"ucker
embedding, is realized as the graded subring of $\cO(\rSL_n)$
generated by the determinants $d_I$ of the minors obtained by taking
(distinct) rows $I=(i_1, \dots, i_r)$ and columns $1, \dots, r$.
In fact one can readily check that $d=\det(a_{ij})_{1 \leq i,j \leq
  r}$ is a classical section and, denoting by $\pi: \cO(\rSL_n)
\lra \cO(P)$  the natural projection
dual to the inclusion $P \subset \rSL_n$, that
$$
(\text{\it id} \otimes \pi)\Delta(d_I) = d_I \otimes \pi\big(d\big)~.
$$

In \cite{fi1} the quantum Grassmannian $\cO_q(\Gr)$ 
is defined as the 
graded subring of  $\cO_q(\rSL_n)$ generated by all of the
quantum determinants $D_I$ of the minors obtained by taking
(distinct) rows $I=(i_1, \dots, i_r)$ and columns $1, \dots, r$.
It is a quantum
deformation of $\cO(\Gr)$ and a quantum homogeneous
projective space for the quantum group $\cO_q(\rSL_n)$,   
(see \cite{fi1,fg1} for more details).  Again one can
readily check that $d=D_{1\dots r}$ is a quantum section and that
$$
(\text{\it id} \otimes \pi)\Delta(D_I) = D_I \otimes \pi\big(d\big),
$$
where
$\cO_q(P)=\cO_q(G)/I_q(P)$ is the quantum subgroup of $\cO_q(G)$
defined by the Hopf
$I_q(P)=(a_{ij})$ generated by the elements $a_{ij}$ for $r+1\leq i \leq n$ and $1 \leq j \leq r$,
and $\pi: \cO_q(G) \lra \cO_q(P)$. 
\end{example}

\section{Quantum Principal bundles from parabolic 
quotients $G/P$}\label{QPB}
In the previous section we have seen how to construct a quantum
homogenous projective variety $\tilde{\cO}_q(G/P)$ given a quantum section $d\in
\cO_q(G)$. We here show how quantum sections lead to quantum principal
bundles over  quantum homogeneous projective varieties.

\subsection{Sheaves of comodule algebras}

Let as before $G$ be a semisimple algebraic group, $P$ a parabolic
subgroup.

We start with a classical observation recalling the construction of a (finite) basis $\{t_i\}_{i\in
{\cal I}}$ of the module of global sections of the very ample line bundle ${\cal L}\to G/P$
associated with a classical section $t\in \cO(G)$.  We also construct the
corresponding  open cover $\{V_i\}_{i\in {\cal I}}$ of
$G$.
\begin{observation}\label{class-gen}
Recalling Proposition \ref{t}, we consider 
an element in $t\in \cO(G)$ satisfying \eqref{eq2.2} and \eqref{eq2.3} and
defining a very ample line bundle ${\cal L}\to G/P$, with $t\in
\tilde{\cO}(G/P)_1\subset \cO(G)$ that is
now a section of ${\cal L}$. Let  $\Delta(t)=\sum t_{(1)}
\otimes t_{(2)}=\sum_{i\in {\cal I}} t^i\otimes t_i$ be its coproduct and notice that the elements
$t_i$ can be chosen to be linearly independent. We now show
that $\{t_i\}_{i\in {\cal I}}$ is a basis of
$\tilde{\cO}(G/P)_1$, the module of global section of $\cL$, hence the $t_i$'s generate 
$\tilde{\cO}(G/P)$ as a (graded) algebra. Indeed, 
by the Borel-Weyl-Bott theorem, 
$\tilde{\cO}(G/P)_1$ is an irreducible $G$ module (corresponding to
the infinitesimal weight  uniquely associated to $\chi$).
By the very definition of $\Delta$, the $G$-action on $t$ is given by,
for all $g,x \in G$:
\begin{equation} \label{lincomb}
(g\cdot t)(x)=t(g^{-1}x)=\Delta(t)(g^{-1} \otimes x)=
\sum  t^i(g^{-1})\, t_i(x)~.
\end{equation}
Since $\tilde{\cO}(G/P)_1$ is irreducible, for any $f \in \tilde{\cO}(G/P)_1$
there exists a $g \in G$, such that $f=g \cdot t$ and
consequently $f$ is a linear combination of the $t_i$'s by
(\ref{lincomb}).
Hence the $t_i$'s form a basis of $\tilde{\cO}(G/P)_1$.

Furthermore, a covering of $G$ is given
by $\{V_i\}_{i\in {\cal I}}$,  where the open sets $V_i$
are  defined by the non vanishing of the corresponding $t_i\in \cO(G)$.
This is so because the line bundle $\cL$ defines a projective embedding
of $G/P$, hence there are no common zeros for its global sections.
\end{observation}

Based on the previous observation we have the following important
property of the quantum homogeneous projective variety  $\tilde{\cO}_q(G/P)$.

\begin{lemma} \label{def-di}
Let $d$ be a quantum section, and $\Delta(d)=\sum d_{(1)} \otimes
d_{(2)}=\sum_{i\in {\cal I}} d^i\otimes d_i$ be its coproduct.
Then the $d_i$'s can be chosen so to form a basis of  $\tilde{\cO}_q(G/P)_1$ as 
$\kk_q$ free module, hence of
$\tilde{\cO}_q(G/P)$ as graded algebra.
\end{lemma}

\begin{proof}
The fact that the $d_i$'s belong to $\tilde{\cO}_q(G/P)_1$ is non trivial, but
it is an immediate consequence of Proposition 3.10 in \cite{cfg}.
The property that they generate $\tilde{\cO}_q(G/P)_1$ as 
$\kk_q$ free module is a consequence of the same property being true
in the classical setting (see Observation \ref{class-gen}) and comes
through the application of Proposition 1.1 in \cite{gl} followed
by Lemma 3.10 in \cite{fh}. The last property immediately follows from
the assumption that $\tilde{\cO}_q(G/P)$ is generated by
$\tilde{\cO}_q(G/P)_1$.
\end{proof}

\medskip
We assume that $$S_i:=\{d_i^r, \, r \in \Z_{\geq 0}\}$$ is Ore
  in order to consider  localizations of  $\cO_q(G)$ and hence
  define a sheaf. We furtherly assume that $S_i$ is Ore in 
the graded subalgebra $\cO_q(G/P)$ of $\cO_q(G)$.
We can then define:
\beq
\cO_q(V_i):=\cO_q(G)S_i^{-1}~,
\eeq
the Ore extension of $\cO_q(G)$ with respect to the multiplicatively
closed set $S_i$. Notice that $\cO_q(V_i)$ is a quantization of
$\cO(V_i)$, the coordinate ring of the open set $V_i \subset
G$.

\begin{proposition}\label{comod1}
The algebra $\cO_q(V_i)$ is an $\cO_q(P)$-comodule algebra
with coaction
$\de_i:\cO_q(V_i) \lra \cO_q(V_i) \otimes \cO_q(P)$ given by:
\beq\label{comod1-eq}
\de_i(x)=((id \otimes \pi) \circ \Delta)(x), \qquad
\de_i(d_i^{-1})=d_i^{-1} \otimes \pi(d)^{-1}, \qquad x \in \cO_q(G)
\eeq
where with an abuse of notation we write $\pi(d)^{-1}$ for the antipode of
$\pi(d)$ 
in $\cO_q(P)$. 
\end{proposition}

\begin{proof} 
Notice that $\cO_q(G)$ is an $\cO_q(P)$-comodule algebra with coaction
$\Delta_\pi=(\mathrm{id} \otimes \pi)\circ \Delta$. Since
$\Delta_\pi(d_i)= d_i \otimes \pi(d)$ is invertible in
$\cO_q(V_i)\otimes \cO_q(P)$ by the universality of the
Ore construction we have our definition of $\de_i$.
\end{proof}

Assume now we can form iterated Ore extensions:
\beq\label{order-ore}
\cO_q(V_{i_1} \cap \dots \cap V_{i_s}):=
\cO_q(\cap_{i \in I} V_i):= \cO_q(G)S_{i_1}^{-1}\dots S_{i_s}^{-1},
\qquad I=\{i_1, \dots, i_s\}
\eeq
{\it independently from the order}, i.e. we assume to have a natural isomorphism
between $\cO_q(V_i\cap V_j)$ and $\cO_q(V_j\cap V_i)$. This is in general a
very restrictive hypothesis, neverthless we will see it is verified
in some interesting examples in the next section.

We also define:
\beq\label{restr}
r_{IJ}: \cO_q(\cap_{i \in I} V_i) \lra \cO_q(\cap_{j \in J} V_j),
\qquad I \subset J
\eeq
as the natural morphism obtained from the Ore extension.

\medskip
Setting as usual $V_I=\cap_{i \in I} V_i$
 we immediately have the following proposition
(cf. Proposition \ref{comod1}).

\begin{proposition} \label{comod2}
$\cO_q(V_I)$ is an $\cO_q(P)$-right comodule algebra
and the morphisms $r_{IJ}$ are $\cO_q(P)$-right comodule algebra morphisms.
\end{proposition}

Let us now consider the opens $U_I:=\wp(V_I)$, obtained via the
projection $\wp:G \lra G/P$. We have the following.

\begin{proposition} \label{shf-def-prop}
The assignment:
$$
U_I \mapsto \cF(U_I):=\cO_q(V_I)~,
$$
with the restriction maps $r_{IJ}:\cO_q(V_I)\rightarrow \cO_q(V_J)$, defines a sheaf of $\cO_q(P)$-comodule algebras on 
$G/P=\cup_{i \in \cI} U_i$, and more in general on $M:=\cup_{i \in
  \cJ} U_i\subset G/P$, where  $I\subset\cI$ and
$I\subset \cJ\subset \cI$ respectively.
\end{proposition}

\begin{proof}
The opens $U_I$ with $I\subset \cI$ (and the empty set) form a basis $\cB$ for a topology on
$G/P$. Recalling Proposition \ref{sheaf-prop} we just have to show that the
assignment $U_I \mapsto \cF(U_I):=\cO_q(V_I)$, with the restriction
maps $r_{IJ} $, defines a $\cB$-sheaf of $\cO_q(P)$-comodule algebras. Since restrictions morphisms are
actually algebra inclusions, 
using the existence of iterated
Ore extension and their compatibility this is straighforwardly seen to
be a
$\cB$-sheaf of algebras and of $O_q(P)$-comodule
algebras.

The sheaf on the more general open submanifold $M=U_{i\in \cJ}U_i$ is simply obtained by considering the opens  $U_I$ with
$I\subset \cJ\subset \cI$.
\end{proof}

\subsection{Quantum principal bundles on quantum homogeneous
spaces} \label{qpb-sec}

In the previous section we have constructed a sheaf of
comodule algebras $\cF$ on $M \subset G/P$. We now want to
define a quantum ringed space structure on the topological
space $M$ as in Definition \ref{qringed-def} and show that $\cF$ is
a quantum principal bundle on it. Notice that $M$ coincides
with $G/P$ if $\cJ=\cI$, while  for  $\cJ\varsubsetneq \cI$, i.e. for a proper subset of the
set of indices $\cI$ of the open cover $\{V_i\}_{i\in {\cal I}}$ of $G$, we have that $M$
is a proper open subset of $G/P$.
\\

By Observation \ref{class-gen}
we know that  $\{U_i:=\wp(V_i)\}_{i\in \cI}$ is an open cover of $G/P$.
Define $\cO_q(U_i)$ as the subalgebra of $\cO_q(G)S_i^{-1}$ generated by
the elements $d_kd_i^{-1}$, for $\kk\in \cI$:
$$
\cO_q(U_i):=
k_q[ d_kd_i^{-1} ]^{\phantom{M_{J_J}}}_{{k\in \cI}} \subset  \cO_q(G)S_i^{-1}~.
$$
Because of our (graded) Ore hypothesis, this is also the 
subalgebra of elements
of degree zero inside $\tilde{\cO}_q(G/P)S_i^{-1}$ and, for this reason, it is called
the (noncommutative)
\textit{projective localization} of $\tilde{\cO}_q(G/P)$ at $S_i$.

\begin{proposition}\label{qringed-prop}
Let the notation be as above. The assignment
$$
U_I \mapsto \cO_q(U_I)
$$
defines a sheaf $\cO_M$ on $M=\cup_{i \in\cJ}U_i$, 
hence $(M,\cO_M)$ is a quantum ringed
space.
\end{proposition}

\begin{proof}
According to Proposition \ref{sheaf-prop} it is enough to check 
that our assignment is a $\cB$-sheaf for the basis associated with the
opens $\{U_i\}$, 
but this is immediate by our hypothesis on the existence of iterated
Ore extension and their compatibility.
\end{proof}

\begin{proposition}\label{coinv-prop}
Let the notation be as above. Then
$\cF(U_i)^{\coinv\, \cO_q(P)}=\cO_M(U_i)$, i.e.
it is the subring in $\cF(U_i)$ generated by
the elements $d_jd_i^{-1}$.
\end{proposition}

\begin{proof} By our definition of coaction $\de_i$ (see (\ref{comod1-eq}))
$$
\de_i(d_jd_i^{-1})=(d_j \otimes \pi(d))(d_i^{-1} \otimes  \pi(d)^{-1})=
d_jd_i^{-1} \otimes 1~.
$$
We now need to prove that the $d_jd_i^{-1}$ generate the subring of 
coinvariants.
Assume $z \in \cF(U_i)^{\coinv\, \cO_q(P)}\subset \cF(U_i)$. Since
$\cF(U_i):=\cO_q(G)[S_i^{-1}]$, then 
$zd_i^r \in \cO_q(G)$ for a suitable $r$. Notice that:
$$
\de_i(zd_i^r )=(z \otimes 1)(d_i^r \otimes \pi(d)^r)=
zd_i^r \otimes \pi(d)^r~.
$$
Hence $zd_i^r \in \tilde{\cO}_q(G/P)_r$, 
which, by Lemma \ref{def-di}, is generated by the $d_j$'s:
$$
zd_i^r= \sum_{\lambda_{j_i \dots j_r} \in k_q} 
\lambda_{j_i \dots j_r}d_{j_1} \dots d_{j_r}~.
$$
Therefore we have: 
$$
z= \sum_{\lambda_{j_i \dots j_r} \in k_q} 
\lambda_{j_i \dots j_r}d_{j_1} \dots d_{j_r}d_i^{-r}~.
$$
We now proceed by induction on $r$. The case $r=0$ is clear. For generic $r$,
since $d_i$ satisfies the Ore condition:
$$
d_{j_r}d_i^{-(r-1)}=d_i^{-(r-1)}\!\sum_{\mu_{j_rs} \in k_q} \mu_{j_r s} d_s~,
$$
hence:
$$
z= \sum_{\lambda_{j_i \dots j_r} \in k_q} 
\lambda_{j_i \dots j_r}d_{j_1} \dots d_{j_{r-1}}d_i^{-(r-1)}\!\sum_{\mu_{j_rs} \in 
  k_q}\mu_{j_rs}d_{s}d_i^{-1}
~.
$$
By induction  we obtain:
$$
z= \sum_{\nu_{j_i \dots j_r} \in k_q} 
\nu_{j_i \dots j_r}d_{j_1}d_i^{-1} \dots d_{j_{r-1}}d_i^{-1}\!\sum_{\mu_{j_rs} \in 
  k_q}\mu_{j_rs}d_{s}d_i^{-1}
$$
hence our result.
\end{proof}

\medskip
We conclude summarizing the main results we have
obtained.

\begin{theorem}\label{main}
Let $G$ be a semisimple algebraic group and $P$ a parabolic subgroup,
let the quantum group $\cO_q(G)$ and the quantum subgroup $\cO_q(P):=\cO_q(G)/I_q(P)$ be the quantizations 
of the coordinate rings $\cO(G)$ and $\cO(P)$.
Let $d$ be a quantum section (see Definition \ref{qsec}), denote with $\{d_i\}_{i\in \cI}$ a choice of
linearly independent elements in the coproduct $\Delta(d)=\sum_{i\in
  \cI} d^i \otimes d_i$, and assume they generate the homogenous
coordinate ring $\tilde{\cO}_q(G/P)$ (see Lemma \ref{def-di}). Assume furtherly that
$\cO_q(V_i):=\cO_q(G)S_i^{-1}$, $S_i=\{d_i^r, r \in \Z_{\geq 0}\}$ is Ore and that
subsequent localizations do not depend on the order (see (\ref{order-ore})).
Then:
\begin{enumerate}
\item Let $\cO_q(U_i):=
k_q[d_kd_i^{-1}]^{\phantom{M_{L_J}}}_{k\in \cI} \subset  \cO_q(G)S_i^{-1}$. The assignment
$U_i \mapsto \cO_q(U_i)$
defines a sheaf $\cO_M$ on $M=\cup_{i \in \cJ}U_i$, $\cJ\subset \cI$, 
hence $(M,\cO_M)$ is a quantum ringed
space. \\
\phantom{JJ} In particular, for $M=G/P$ (${\cal J}={\cal I}$), 
the sheaf $\cO_{G/P}$ is the projective localization of the
homogeneous coordinate ring $\tilde{\cO}_q(G/P)$.

\item
The assignment:
$U_I \mapsto \cF(U_I):=\cO_q(V_I)$
defines a sheaf $\cF$ of $\cO_q(P)$-comodule algebras on the quantum ringed space
$M=\cup_{i \in \cJ} U_i\subset G/P$.

\item  $\cF^{\coinv\, \cO_q(P)}=\cO_M$, i.e.,  the subsheaf $\cF^{\coinv\,
    \cO_q(P)}: U\to \cF(U)^{\coinv\,
    \cO_q(P)}\subset \cF(U)$ is canonically 
isomorphic to the sheaf $\cO_M$.
\end{enumerate}
If the sheaf $\cF$ is locally cleft (see Definition \ref{hg-shf})
then $\cF$ is a quantum principal bundle. 
\end{theorem}

\begin{proof}
(1) is Proposiiton \ref{qringed-prop}. (2) is Proposition \ref{shf-def-prop}.
(3) is Proposition \ref{coinv-prop}.
\end{proof}

\section{Examples}
\label{ex-sec}

In this section we apply the general theory we have developped and present
quantum principal bundles over quantum projective spaces. 
We hence sharpen the notion of quantum projective space  as quantum 
homogenous space.
In this section the ground field is $k=\mathbb{C}$.

\subsection{Quantum deformations of function algebras}

\label{qdefs-subsec}
We start with an important example of
quantum group and its quantum homogeneous varieties. For
more details see \cite{ma1} and \cite{fi1}.

\begin{definition} \label{qmatrices}
We define the \textit{quantum matrices}  
as the $\C_q$ algebra $\cO_q(M_{n})$:
\begin{equation} \label{qmatrices-eq}
\cO_q(\M_n) \, = \, \C_q \langle a_{ij} \rangle /I_{\M}
\end{equation}
where $i,j=1,\dots n$ and $I_{\M}$ is the ideal of the Manin relations:
\begin{equation}\label{manin}
\begin{array}{c}
a_{ij}a_{kj}=q^{-1}a_{kj}a_{ij} \quad i<k\,, \quad \quad 
a_{ij}a_{kl}=a_{kl}a_{ij}\quad i<k,j>l \quad \!\!\!\mbox{or}\!\!\! \quad i>k,j<l\,,\\[.6em]
a_{ij}a_{il}=q^{-1}a_{il}a_{ij} \quad j<l\,,  \quad \quad 
a_{ij}a_{kl}-a_{kl}a_{ij}=(q^{-1}-q)a_{il}a_{kj} \quad i<k,j<l_{}\,.
\end{array}
\end{equation}
The quantum matrix algebra $\cO_q(\M_{n})$ is a bialgebra, with
comultiplication and counit given by:
$$
\Delta(a_{ij})=\sum_k a_{ik} \otimes a_{kj}, 
\qquad 
\epsi(a_{ij})=\de_{ij}.
$$
We define the \textit{quantum general linear group} to be the algebra
$$
\cO_q(\rGL_n)= \cO_q(\M_n)[\ddet_q^{-1}]
$$
where $\ddet_q$ is the \textit{quantum
  determinant}:
$$
\ddet_q(a_{ij})=\sum_\sigma (-q)^{-\ell(\sigma)}a_{1 \sigma(1)} 
\dots a_{n \sigma(n)}
=\sum_\sigma (-q)^{-\ell(\sigma)}a_{\sigma(1)1} \dots a_{\sigma(n)n}
$$
where $\ell(\sigma)$ is the length of the permutation $\sigma$
(see \cite{pw} for more details on quantum determinants).

We define the \textit{quantum special linear group} to be the algebra
$$
\cO_q(\rSL_n)= \cO_q(M)/(\ddet_q-1)
$$
$\cO_q(\rGL_n)$ and $\cO_q(\rSL_n)$ are Hopf algebras and quantum deformations
respectively of the general linear and the special linear groups.
\end{definition}

\subsection{Quantum principal bundles on 
quantum Projective spaces}

We consider the special case of a maximal parabolic subgroup 
$P$ of $G=\rSL_n(\C)$ of the form:
$$
P=\left\{\begin{pmatrix} p_{11} & p_{12} & \dots & p_{1n} \\
0& p_{22} & \dots & p_{2n} \\
\vdots & & & \vdots\\
0 & p_{n2} & \dots & p_{nn} \end{pmatrix} \right\}\,\subset\, G=
\left\{A=\!\begin{pmatrix} a_{11} & \dots & a_{1n} \\
a_{21} & \dots & a_{2n} \\
\vdots  & & \vdots\\
a_{n1}  & \dots & a_{nn} \end{pmatrix}\!\!, \det(A)=1
\right\}.
$$
In this case $G/P \simeq \bP^{n-1}(\C)$ is the complex projective space,
and $\tilde{\cO}(\bP^{n-1})$ is the corresponding free graded ring with $n$ generators.
Its quantization  $\tilde{\cO}_q(\bP^{n-1})$ is well known
and, for example, it is constructed in detail
in \cite{fi1} (see Theorem 5.4 for $r=1$), see also \cite{dl}. $\tilde{\cO}_q(\bP^{n-1})$  is  the 
subring  of $\cO_q(\rSL_n)$ generated by the
elements $x_{i}=a_{i1}$, $i\in\cI=\{1,...n\}$.
We can immediately give a presentation:
\beq\label{proj-pres}
\tilde{\cO}_q(\bP^{n-1})=\C_q\langle x_1, \dots , x_{n}\rangle/(x_ix_j-q^{-1}x_jx_i, i < j)~. 
\eeq

We reinterpret  this construction  within the present framework, first
showing that $\tilde{\cO}_q(\bP^{n-1})$ is a quantum homogeneous
projective space according to Definition \ref{qhv-def} and then
constructing, along Theorem \ref{main}, an $\cO_q(P)$-principal bundle on
the ringed space obtained via projective localizations of  $\tilde{\cO}_q(\bP^{n-1})$.

\medskip

Let $\cO_q(G)=\cO_q(\rSL_n)$ be the quantum special linear group
of Definition \ref{qmatrices},  and define the quantum parabolic subgroup
\begin{equation}\label{qparSL}
\cO_q(P):= \cO_q(\rSL_n)/I_q(P)~,
\end{equation}
 where $I_q(P)=(a_{\al1})$ is the Hopf ideal generated by
$a_{\al1}$, $\al\in \{2,\ldots n\}$. We use coordinates $p_{ij}$ 
for the images of the generators $a_{ij}$ under $\pi: \cO_q(\rSL_n) \lra \cO_q(P)$. 
We notice (cf. Example \ref{qgrass-ex}) that $d=a_{11} \in \cO_q(\rSL_n)$ is a quantum section, in fact
$$
\Delta_\pi(a_{11})=a_{11} \otimes p_{11}, \qquad p_{11}=\pi(a_{11})~.
$$
Furthermore, from the coproduct
$
\Delta(a_{11})=\sum_{i\in {\cal I}} a_{1i} \otimes a_{i1}
$
we choose the linearly independent elements $d_i$ in
$\Delta(d)=\sum_{i\in\cI}d^i\otimes d_i$, to be $$d_i=a_{i1}~.$$
Hence, by Lemma \ref{def-di}, the  elements $d_i$ span   $\tilde{\cO}_q(\rSL_n/P)_1$, as defined in
\eqref{G/P-eq}.
The quantum homogeneous projective variety $\tilde{\cO}_q(\rSL_n/P)$ is generated in
degree one, cf. Example \ref{qgrass-ex}, and one can see immediately that $\tilde{\cO}_q(\rSL_n/P)$
coincides with $\tilde{\cO}_q(\bP^{n-1})$, as defined in (\ref{proj-pres}).
\medskip

We now structure $\tilde{\cO}_q(\bP^{n-1})$ as a quantum ringed space and
construct a sheaf of locally trivial $\cO_q(P)$-comodule
algebras, i.e., a quantum principal bundle on the quantum projective
space $\tilde{\cO}_q(\bP^{n-1})$, where $\cO_q(P)$ is the quantum parabolic
subgroup of $\cO(\rSL_n)$ defined in \eqref{qparSL}.

Let us consider the two classical open covers of the
topological spaces $\rSL_n(\C)$ and $\bP^{n-1}(\C)$ respectively:
\begin{equation} \label{opencover}
\begin{array}{rl}
\rSL_n(\C) &=\cup_i V_i, \qquad V_i=\{g \in \rSL_n(\C) \, | \, a_{i1}^0(g) \neq 0\}\\ \\
\bP^{n-1}(\C) &=\cup_i U_i, 
\qquad U_i=\{z \in \bP^{n-1}(\C) \, | \, x_{i}^0(z)\neq 0\}
\end{array}
\end{equation}
where $a_{ij}^0$ denote the generators of $\cO(\rSL_n)$ and similarly
$x_i^0$ those of $\tilde\cO(\bP^{n-1})$, $i,j=1, \dots, n$. 
Evidently, $\wp(V_i)=U_i$, $\wp:\rSL_n(\C) \lra \rSL_n(\C)/P=\bP^{n-1}(\C)$.

\begin{lemma} 
The multiplicative set $S_i=\{a_{i1}^k\}^{\phantom{M_J}}_{{k \in \N}} \subset \cO_q(\rSL_n)$ 
satisfies the Ore condition. 
Furthermore, 
$\cO_q(\rSL_n)S_{i_1}^{-1}\dots S_{i_s}^{-1}$,
does not depend on the
order of the Ore extensions.
\end{lemma}

\begin{proof} See \cite[pp. 4 and 5]{sk2}. Notice
that $a_{i1}$ is a quantum minor of order $1$ and two such minors
$q$-commute, hence their product forms an Ore set.
\end{proof}

 As a corollary of Theorem 4.8 we then immediately obtain
\begin{proposition} \label{shf-def-prop2}
Let the notation be as in the previous section.
The assignment:
$$
U_I \longmapsto \cF(U_I):=\cO_q(V_I):=\cO_q(\rSL_n)S_{i_1}^{-1}\dots S_{i_s}^{-1},
\qquad I=\{i_1, \dots, i_s\}
$$
defines a sheaf of $\cO_q(P)$-comodule algebras on $\rSL_n(\C)/P$.
Furthermore, $\cF(U_i)^{\coinv \,\cO_q(P)}$ is generated by
$a_{i1}a_{11}^{-1}$, $\cF^{\coinv \,\cO_q(P)}$ equals the projective
localization of $\tilde{\cO}_q(\bP^{n-1})$ and $(\rSL_n(\C)/P, \cF^{\coinv \, \cO_q(P)})$ is a quantum ringed space.
\end{proposition}

We now show that $\cF$ is a quantum principal
bundle on the quantum ringed space $(\rSL_n(\C)/P, \cF^{\coinv \,
  \cO_q(P)})$.
The only property to be checked is the locally cleft condition
(cf. Definition \ref{cleft-def}). We actually show the stronger local
triviality condition, i.e., the collection of maps $j_i:
\cO_q(P)\to \cF(U_i)$ are  
$\cO_q(P)$-comodule algebra maps, hence, in particular, are cleaving
maps  (cf. Remark \ref{cleft-rem}).

\medskip
We first study the map $j_1$.
Let $a_{ij}\in \cF(U_1):=\cO_q(\rSL_n)S_{1}^{-1}= \cO_q(\rSL_n)[a^{-1}_{11}]$, ${i,j=1,...n}$; since $a_{11}$ is invertible we have the matrix factorization 
\begin{equation}\label{factor}
(a_{ij})=\begin{pmatrix} 1 &0\\a_{\alpha 1}a^{-1}_{11}& 1\!\!1\end{pmatrix}\!\!
\begin{pmatrix} a_{11} & a_{1\beta}~\\ 0& a_{\alpha \beta }-a_{\alpha 1}a^{-1}_{11}a_{1\beta}\end{pmatrix}
=\begin{pmatrix} 1 &0\\a_{\alpha 1}a^{-1}_{11}& 1\!\!1\end{pmatrix}\!\!
\begin{pmatrix} a_{11} & a_{1\beta}\\ 0& a_{11}^{-1}
  D^{1\beta}_{1\alpha}\end{pmatrix}_{_{}}
\end{equation}
where 
$\alpha, \beta=2,\ldots n$, and
$D_{ij}^{kl}=a_{ik}a_{jl}-q^{-1}a_{il}a_{jk}$, with $i<j$ and $k<l$, denotes the quantum determinant of the $2\times 2$
quantum matrix obtained by taking rows $i,j$ and columns $k,l$.

In the commutative case this factorization corresponds to the
trivialization $V_1\simeq \C^{n-1}\times P$ of the
open $V_1$ of the total space of  $\rSL_n(\C)\rightarrow
\rSL_n(\C)/P$ 
(cf. eq. \eqref{opencover}).  In the quantum case we similarly
have that $\cF(U_1)^{\coinv \,\cO_q(P)}\subset \cF(U_1)$ is a trivial
Hopf-Galois extension. Recalling
Remark \ref{cleft-rem} and Observation \ref{obsjtheta}, we shall see it is the smashed
product $$\cF(U_1)=\C_q[a_{\alpha 1} a_{11}^{-1}]_{\al=2,...n}\,\#\,
\cO_q(P)~,$$ where the generators $\Big(\small{\begin{matrix} p_{11} &
    \!\!p_{1\beta}\\ 0& \!\! 
  p_{\alpha\beta}\end{matrix}_{_{}}}\Big)$ of $\cO_q(P)$ are 
identified with $\Big(\small{\begin{matrix} a_{11} &
    \!\!a_{1\beta}\\ 0& \!\! a_{11}^{-1}
  D^{1\beta}_{1\alpha}\end{matrix}_{_{}}}\Big)$.

The properties of $j_1: \cO_q(P)\to \cF(U_1)$ follow from the
properties of an associated lift $J_1$ that maps into the localization
$\cO_q(\M_n)[a_{11}^{-1}]$ of the quantum matrix algebra defined in \eqref{qmatrices-eq}.

\begin{lemma}\label{lemmaqpb1}
Let $\cO_q(p_{ij})$ denote the quantum matrix algebra
with generators $p_{ij}=p_{11}, p_{1\beta}, p_{\al\beta}$ and $p_{\al 1}=0$; $\al,\beta=2,\dots,n$.
We have a well defined algebra map $J_1:\cO_q(p_{ij}) \lra \cO_q(\M_n)[a_{11}^{-1}]$,
that on the generators reads
$$
\begin{array}{c}
J_1(p_{11}^{\pm 1})=a_{11}^{\pm 1}\;, \qquad J_1(p_{1\beta})=a_{1\beta}\;, 
\quad J_1(p_{\alpha\beta})=a_{11}^{-1} D^{1\beta}_{1\alpha}\;\,. 
\end{array}
$$
\end{lemma}

\begin{proof}
  Recall, from \cite{fi3}, the following commutation relations in
  $\cO_q(\M_n)$ among  quantum determinants and generators of the algebra
  of quantum matrices:
$$
\begin{array}{c}
a_{1\beta}D_{1\al}^{1\beta}=D_{1\al}^{1\beta}a_{1\beta},\quad
a_{1\ga}D_{1\al}^{1\beta}=qD_{1\al}^{1\beta}a_{1\ga}, \quad \ga>\beta \\ \\
a_{1\ga}D_{1\al}^{1\beta}=qD_{1\al}^{1\beta}a_{1\ga}+q(q^{-1}-q)D^{1\ga}_{1\beta}a_{1\al}, \quad \ga<\beta 
\end{array}
$$
where $D_{1\al}^{1\beta}=a_{11}a_{\al\beta}-q^{-1}a_{1\beta}a_{\al1}$.
Also, by  Theorem 7.3 in \cite{fi3}, the indeterminates $u_{\al\beta}:=
D_{1\al}^{1\beta}$
satisfy the Manin relations as in Definition \ref{qmatrices},
where we replace $a_{\al\beta}$ with $u_{\al\beta}$.
In order to show that $J_1$ is an algebra map, we  have to show it
is well defined.  First, we easily compute the
commutation relations $$a_{11}^{\pm 1}
D^{1\beta}_{1\alpha}=D^{1\beta}_{1\alpha}a_{11}^{\pm 1}~,$$
that imply that
the $a_{11}^{-1} D^{1\beta}_{1\alpha}$'s satisfy the Manin relations
among themselves.
Next, we need to check that the commutation relations between
$a_{1\gamma}$, $\gamma=2,\ldots n$, and $a_{11}^{-1} D^{1\beta}_{1\alpha}$
are of the Manin
kind. 

If $\gamma >\beta$, we have:
$$
a_{1\gamma\,} ~ a_{11}^{-1} \!D^{1\beta}_{1\alpha}\,=
\,a_{11}^{-1} \!D^{1\beta}_{1\alpha\,} ~ a_{1j}
$$
because $a_{1\gamma}a_{11}^{-1}=
q^{-1}a_{11}^{-1}a_{1\gamma}$ and $a_{1\gamma}D^{1\beta}_{1\alpha}=
qD^{1\beta}_{1\alpha}a_{1\gamma}$.

If $\gamma=\beta$, we have:
$$
a_{1\beta\,} ~ a_{11}^{-1} \! D^{1\beta}_{1\alpha}=
q^{-1} a_{11}^{-1} \!D^{1\beta}_{1\alpha\,}~
a_{1\beta}
$$
because $a_{1\beta}$ and $D^{1\beta}_{1\alpha}$ commute. 

If $\gamma<\beta$, we need to check the commutation:
\beq
\begin{array}{rl}
  a_{1\gamma} ~ a_{11}^{-1}\!  D^{1\beta}_{1\alpha}~=~a_{11}^{-1}\!
  D^{1\beta}_{1\alpha}~ a_{1\gamma}+
(q^{-1}-q)a_{1\beta}~a_{11}^{-1}\! D^{1\gamma}_{1\alpha}~.
\end{array}
\eeq
We leave this calculation as an exercise.
\end{proof}

\begin{lemma}\label{lemmaqpb1bis}
  Let the notation be as above. Let
  $\ddet_q(p_{ij})$ and $\ddet_q(a_{ij})$ denote
  respectively the quantum determinants in $\cO_q(p_{ij})$
  and $\cO_q(\M_n)[a_{11}^{-1}]$.
  Then
  $$
  J_1(\ddet_q(p_{ij}))=\ddet_q(a_{ij})~.
  $$
\end{lemma}

\begin{proof}
In the factorization (\ref{factor}), define:
$$
(b_{ij}):= \begin{pmatrix} 1 &0\\a_{\al 1}a^{-1}_{11}& 1\!\!1\end{pmatrix}\qquad
(c_{ij}):=\begin{pmatrix} a_{11} & a_{1\beta}\\ 0& a_{11}^{-1}
D^{1\beta}_{1\al}\end{pmatrix}_{_{}} 
$$
for $i,j=1,\dots, n$, $\al,\beta=2,\dots, n$.
Since $c_{ij}=J_1(p_{ij})$, by Lemma \ref{lemmaqpb1}, they
form a quantum matrix and
our claim amounts to  $\ddet_q(a_{ij})=\ddet_q(c_{ij})$.

\medskip
We start by noticing that $b_{ij}$ and $c_{kl}$ satisfy
the following commutation relations:
\beq\label{comm-rel-bc}
\begin{array}{c}
b_{ij}c_{kl}=c_{kl}b_{ij}, \quad j \neq 1,
\qquad b_{11}c_{kl}=c_{kl}b_{11},  \qquad b_{i1}c_{il}=q^{-1}c_{il}b_{i1},
\quad i>1\\ \\
 b_{i1}c_{kl}=c_{kl}b_{i1}, \quad k<i, \qquad b_{i1}c_{11}=qc_{11}b_{i1} \\ \\
b_{i1}c_{kl}=c_{kl}b_{i1} + (q^{-1}-q)c_{il}b_{k1}, \quad k>i~.
\end{array}
\eeq
We also notice the obvious facts:
\beq \label{obvious}
b_{ii}=1, \qquad b_{ij}=0, \quad i \neq j, \quad j \neq 1~.
\eeq

We proceed with a direct calculation of $\ddet_q(a_{ij})$
using $a_{ij}=\sum_k b_{ik}c_{kj}$. Recall the quantum Laplace expansion
along the first column (see \cite{pw} pg 47):
$$
\ddet_q(a_{ij})=\sum_r (-q)^{-r+1} a_{r1} A(r,1)
$$
where $A(r,1)$ is the quantum determinant obtained
from $(a_{ij})$ by removing the $r$-th row and first column,
\beq\label{expr}
\begin{array}{l}
\ddet_q(a_{ij})= a_{11} \sum_\sigma (-q)^{-\ell(\sigma)}
                 a_{2\sigma(2)} \dots  a_{n\sigma(n)}\\ \\
              \qquad \qquad ~~~+\sum_{t=2}^n (-q)^{1-t} a_{t1}\sum_{\sigma_t} (-q)^{-\ell(\sigma_t)}
a_{1\sigma_t(1)} \dots \widehat{a_{t\sigma_t(t)}} \dots a_{n\sigma(n)} \\ \\
\qquad  \qquad =c_{11} \sum_\sigma (-q)^{-\ell(\sigma)}
b_{2k_2}c_{k_2\sigma(2)} \dots  b_{nk_n}c_{k_n\sigma(n)} \\ \\
\qquad \qquad ~~~+ \sum_{t=2}^n (-q)^{1-t} b_{t1}c_{11}^{-1}
\sum_{\sigma_t,k_1 \dots \widehat{k_t},
\dots k_n}(-q)^{-\ell(\sigma_t)} \\ \\
\qquad \qquad \qquad ~~~~~b_{1k_1}c_{k_1\sigma(1)} \dots \widehat{b_{1k_t}c_{k_t\sigma(t)}} \dots
b_{nk_n}c_{k_n\sigma(n)}~,
\end{array}
\eeq
where $\widehat{u}$ means that we omit the term $u$.

Notice that $\sigma_t:\{1,\dots,\widehat{t},\dots, n\} \lra 
\{2,\dots, n\}$, but we treat it as a permutation, just renaming
the elements of the two sets as the first $n-1$ natural numbers, so
that $\ell(\sigma_t)$ is well defined.

Let us look at the term 
$b_{2k_2}c_{k_2\sigma(2)} b_{3k_3}c_{k_3\sigma(3)} \dots  b_{nk_n}c_{k_n\sigma(n)}$,
where $k_2,\dots,k_n=1,\dots,n$. We want to reorder it, and we claim that:
$$
b_{2k_2}c_{k_2\sigma(2)} 
b_{3k_3}c_{k_3\sigma(3)} \dots  b_{nk_n}c_{k_n\sigma(n)}=
b_{2k_2}b_{3k_3} \dots  b_{nk_n}c_{k_2\sigma(2)} c_{k_3\sigma(3)}
c_{k_n\sigma(n)}~.
$$
By (\ref{obvious}) $b_{2k_2}\neq 0$ if and only if $k_2=1,2$.
So we have to reorder $c_{k_2\sigma(2)} b_{3k_3}$ only for  $k_2<3$, hence, by 
(\ref{comm-rel-bc}), we have that they commute. The rest follows 
by repeated application of this argument.
 
Therefore, we can write the first term in (\ref{expr}) as:
\beq\label{term1}
\begin{array}{l}
\!\! \!\!\!a_{11} \sum_\sigma (-q)^{-\ell(\sigma)}
                 a_{2\sigma(2)} \dots  a_{n\sigma(n)}=\\ \\
\qquad~~~=c_{11} 
\sum_{\sigma,k_2,\dots k_n} (-q)^{-\ell(\sigma)}
b_{2k_2}c_{k_2\sigma(2)} \dots  b_{nk_n}c_{k_n\sigma(n)}\\ \\
\qquad~~~ =c_{11} \sum_{\sigma,k_2,\dots,k_n} (-q)^{-\ell(\sigma)}
b_{2k_2}b_{3k_3} \dots  b_{nk_n}c_{k_2\sigma(2)} \dots c_{k_n\sigma(n)} \\ \\
\qquad~~~ =c_{11} 
\sum_{k_2,\dots,k_n} 
b_{2k_2}b_{3k_3} \dots  b_{nk_n}
  \sum_\sigma (-q)^{-\ell(\sigma)} c_{k_2\sigma(2)} \dots c_{k_n\sigma(n)} \\ \\
\qquad~~~ =c_{11} 
\sum_{k_2,\dots,k_n} 
b_{2k_2}b_{3k_3} \dots  b_{nk_n}C[k_2,\dots,k_n|2,\dots,n]
\end{array}
\eeq
where $ C[k_2,\dots,k_n|2,\dots,n]$ is the quantum determinant
in the indeterminates $c_{ij}$ obtained by taking
rows $(k_2,\dots,k_n)$ (in this order) and columns $(2,\dots,n)$.
Notice that, by (\ref{obvious}), the sum over the index  $k_t$ runs
only on the values $k_t=1$ and $t$. If $k_u=k_v=1$ for some
$u,v=2,\dots,n$, then by Corollary 4.4.4 in \cite{pw}, we have
$C[k_2,\dots,k_n|2,\dots,n]=0$;
so we must have $n-1$ distinct indices $1 \leq k_2,\dots,k_n \leq n$
and $k_u=1$ for at most one of them. 

We rewrite the 
first term in (\ref{expr}) as:
\beq\label{key-expr}
\begin{array}{l}
\!\!\! a_{11} \sum_\sigma (-q)^{-\ell(\sigma)}
a_{2\sigma(2)} \dots  a_{n\sigma(n)} =\\ \\
\qquad ~~~=  c_{11} C[2,\dots,n|2,\dots,n]+c_{11}b_{21}C[1,3, \dots,n|2,\dots,n]\\ \\
  \qquad ~~~~~~+c_{11}b_{31}C[2,1,4 \dots,n|2,\dots,n]
+c_{11}b_{41}C[2,3,1,5\dots,n|2,\dots,n]\\ \\
\qquad ~~~~~~ +\,\dots\, +c_{11}b_{n1}C[2,3 \dots,n-1,1|2,\dots,n] \\ \\
  \qquad ~~~=c_{11} C[2,\dots,n|2,\dots,n]+
  \sum_t (-q)^{2-t}c_{11}b_{t1}
C[1,\dots \widehat{t} \dots,n|2,\dots,n]~.
\end{array}
\eeq

\medskip
Let us now look at the second term in (\ref{expr}).
Reasoning as before, we have:
\beq\label{term2-1}
\begin{array}{l}
\!\!\!(-q)^{1-t} a_{t1}
\sum_\tau (-q)^{-\ell(\tau)} a_{1\tau(1)} \dots \widehat{a_{t\tau(t)}}
\dots a_{n\tau(n)}=~~\\ \\
\qquad\!\!\! =(-q)^{1-t} b_{t1}c_{11}
  \sum_{\tau, k_1,\dots k_t \dots k_n} (-q)^{-\ell(\tau)}
  b_{1k_1}c_{k_1\tau(1)} \dots
\widehat{b_{1k_t}c_{k_t\tau(t)}}\dots b_{nk_n}c_{n\tau(n)}\,.
\end{array}
\eeq
However, we notice that here it must be $k_1=1$, otherwise
$b_{1k_1}=0$, hence this forces $k_t=t$ for all $t>1$.
So we can write:
\beq\label{term2}
\begin{array}{l}
\!\!\!\!\!\!\!\!\!\!\!\!\!\!(-q)^{1-t} a_{t1}
\sum_\tau (-q)^{-\ell(\tau)} a_{1\tau(1)} \dots \widehat{a_{t\tau(t)}}
\dots a_{n\tau(n)}=\\ \\
\qquad~~~~~~~ =(-q)^{1-t} b_{t1}c_{11}
  \sum_\tau (-q)^{-\ell(\tau)}
  c_{1\tau(1)}\dots \widehat{c_{t\tau(t)}}\dots c_{n\tau(n)} \\ \\
\qquad ~~~~~~~= -(-q)^{2-t}c_{11}b_{t1}C[1,\dots \widehat{t},\dots,n|2,\dots,n]
\end{array}
\eeq
because by (\ref{comm-rel-bc}) we have  $b_{t1}c_{11}=qc_{11}b_{t1}$.

If we substitute expressions (\ref{key-expr}) and (\ref{term2})
in (\ref{expr}) and simplify we remain with just one term:
$$
\ddet_q(a_{ij})= c_{11} C[2,\dots,n|2,\dots,n]=\ddet_q(c_{ij})~.
$$\\[-2em]
\end{proof}

\begin{proposition}\label{propqpb1}
The map $j_1:\cO_q(P) \lra
\cF(U_1):=\cO_q(\rSL_n)[a_{11}^{-1}]$ defined on the generators as:
$$
\begin{array}{c}
j_1(p_{11}^{\pm 1})=a_{11}^{\pm 1}\;, \qquad j_1(p_{1\beta})=a_{1\beta}\;, 
\quad j_1(p_{\alpha\beta})=a_{11}^{-1} D^{1\beta}_{1\alpha}\;, 
\end{array}
$$
$\alpha,\beta=2,\ldots  n$, 
 is an $\cO_q(P)$-comodule algebra map.
\end{proposition}

\begin{proof}
We canonically have $\cO_q(\rSL_n)[a_{11}^{-1}]=
 $ $\cO_q(\M_n)[a_{11}^{-1}]/(\det_q(a_{ij})-1)$ and $\cO_q(P)=
 \cO_q(\rSL_n)/I_q(P) =\cO_q(p_{ij})/ (\ddet_q(p_{ij})-1)$ as algebras.
 Because
of the previous lemma, $j_1:\cO_q(P)\to \cO_q(\rSL_n)[a_{11}^{-1}]$ is well defined; in fact it is the
  algebra map $J_1:\cO_q(p_{ij})\to \cO_q(\M_n)[a_{11}^{-1}]$ induced on the
  quotients.

We next show that $j_1$ is an $\cO_q(P)$-comodule morphism, i.e.,
$\de_1 \circ j_1 = ( j_1 \otimes \mathrm{id}) \circ  \Delta_P$,
where $\Delta_P$ is the comultiplication in $\cO_q(P)$ and $\de_1$
is the $\cO_q(P)$ coaction on $\cF(U_1)=\cO(V_1)$ as defined in  Proposition \ref{comod1}. 
Since $j_1$ is an algebra map, it is enough to check the comodule
property on the generators. 
Let us look at the case of $p_{\alpha\beta}$, the case $p_{1j}$
being an easy calculation.
On the one hand,  using the coproduct formula for quantum minors  (see e.g. \cite{fi3})
$$
\Delta(D^{1j}_{1i})=\mbox{$\sum_{r<s}$}\, D^{rs}_{1i} \otimes D^{1j}_{rs}~, 
$$ we have:
\beq\label{comod-sl}
\begin{array}{rl}
\!\!\!\!\!\!(\de_1 \circ j_1)(p_{\alpha\beta})\!\!\!&=\,
\de_1(a_{11}^{-1})\, \de_1(D^{1\beta}_{1\alpha})=
\big(a_{11}^{-1}\otimes \pi(a_{11}^{-1})\big)\sum_{r<s} D^{rs}_{1\alpha} \otimes 
\pi(D^{1\beta}_{rs}) \\[.8em]
&=\,\sum_{k<\gamma} a_{11}^{-1} \!D^{r\gamma}_{1\alpha} \otimes \pi(a_{11}^{-1} \!D^{1\beta}_{r\gamma})= 
\sum_{\gamma} a_{11}^{-1}  \!D^{1\gamma}_{1\alpha}\otimes p_{\gamma\beta}\,.
\end{array}
\eeq
On the other hand:
$$
\begin{array}{rl}
\big(( j_1 \otimes \mathrm{id}) \circ  \Delta_P\big)(p_{\alpha\beta})&=
( j_1 \otimes \mathrm{id}) \sum_\gamma p_{\alpha \gamma} \otimes p_{\gamma\beta}=
 \sum_\gamma a_{11}^{-1} \!D^{1\gamma}_{1\alpha} \otimes p_{\gamma\beta}\,.
\end{array}
$$\\[-2em]
\end{proof}

We now extend the previous proposition in order to define the  
$\cO_q(P)$-comodule algebra  maps
$j_k:\cO_q(P) \to \cF(U_k)=\cO_q(\rSL_n)[a_{\alpha1}^{-1}]$,
($k=1,\ldots n$) 
thus proving the triviality of the Hopf-Galois extensions
$\cF(U_k)^{\coinv \,\cO_q(P)}\subset \cF(U_k)$. 

Reasoning as before, for each fixed value of $k$, 
we consider the factorization of quantum matrices $(a_{ij})$
similar to (\ref{factor}):
$$\small{
\begin{pmatrix}  
a_{11}a_{k1}^{-1} & 0 & 0&\ldots &  0 & 1&\;0 &\dots &0 \\
a_{21}a_{k1}^{-1} & 1 & 0& \ldots &  0 & 0&\;0 &\dots &0\\
a_{31}a_{k1}^{-1} & 0 & 1&\ldots &  0 & 0 & \;0&\dots & 0\\
\vdots  & \vdots & &\ddots & &\vdots & & \\
a_{k-1_{\:\!}1}a_{k1}^{-1} & 0 & 0 & \ldots &1 &0&  \;0 &\ldots & 0 \\[.8em]
1 & 0 &  \ldots & & 0 & 0 &0 & \ldots&0\\[.8em]
\!a_{k+1{\:\!}1}a_{k1}^{-1} & 0 & \ldots &   & 0 &0& \;1 & 0~ \;\ldots  &0
  \\
\!a_{k+2{\:\!}1}a_{k1}^{-1} &0 & \ldots & &  0 &  0&0 &1~\;\ldots &  0 \\
\vdots  & & & & &  &\vdots & ~~\ddots\!\!\! &\vdots\\
a_{n1}a_{k1}^{-1} & 0 & \dots & & 0& 0& 0 &~~\ldots\!\!\! & 1 \end{pmatrix}\!\!
\begin{pmatrix} a_{k1}\!\! & a_{k\beta}  \\[.2em]
0\!\!& a_{2\:\!\beta}-a_{2\:\! 1}a^{-1}_{k\:\!1}a_{k\beta}  \\[.2em]
0\!\!& a_{3\:\!\beta}-a_{3\:\! 1}a^{-1}_{k\:\!1}a_{k\beta}  \\[.2em]
\vdots  & \vdots\\
0\!\!& a_{k-1\:\!\beta}-a_{k-1\:\! 1}a^{-1}_{k-1\:\!1}a_{k\beta} \\[.2em]
0\!\!& a_{1\:\!\beta}-a_{1\:\! 1}a^{-1}_{k\:\!1}a_{k\beta}  \\[.2em]
0\!\!& a_{k+1\:\!\beta}-a_{k+1\:\! 1}a^{-1}_{k\:\!1}a_{k\beta}  \\[.2em]
0\!\!& a_{k+2\:\!\beta}-a_{k+2\:\! 1}a^{-1}_{k\:\!1}a_{k\beta}  \\[.2em]
\vdots\!\! &  \vdots\\
0\!\!& a_{n\:\!\beta}-a_{n\:\! 1}a^{-1}_{k\:\!1}a_{k\beta}  \\
\end{pmatrix}}
$$
where 
$\beta=2,\ldots n$.
This suggests to exchange row $k$ with row $1$ in order
to identify the
last matrix with the matrix of  generators $\Big(\small{\begin{matrix} p_{11} &
    \!\!p_{1\beta}\\ 0& \!\!  p_{\alpha\beta}\end{matrix}_{_{}}}\Big)$ of 
$\cO_q(P)$. 

\begin{proposition}\label{lemmaqpb2}
The map $j_k:\cO_q(P) \lra \cF(U_k)=\cO_q(\rSL_n)[a_{k1}^{-1}]$, defined on the generators as:
$$
\begin{array}{c}
j_k(p_{11}^{\pm 1})=a_{k1}^{\pm 1}, \quad j_k(p_{1\beta})=a_{k\beta}, 
\quad j_k(p_{\alpha\beta}) = \begin{cases}
-q^{-1}(a_{1\beta}-a_{11} a_{k1}^{-1}a_{k\beta})~ \al=k\\
a_{\al\beta}-a_{\al 1}a_{k1}^{-1}a_{k\beta} \quad\quad\quad \al\not= k
\end{cases} \!\!\!\!,\\[1.4em]
\!\!\!\!\!\!\!\!\!\!\!\!\!\!\!\!\!\!\!\!\!\!\!\!\!\!\!\!\!\!\!\!\!\!\!\!\!\!\!\!\!\!\!\!\!\!\!\!\!\!\!\!\!\!\!\!\!\!\!\!\!\!\!\!\!\mbox{\it i.e., equivalently, }~
j_k(p_{\alpha\beta}) = \begin{cases} -qD^{1\beta}_{\al k}a_{k1}^{-1},
  \quad \al <k \\ 
D^{1\beta}_{1k}a_{k1}^{-1}, ~~~\quad \al =k \\  
D^{1\beta}_{k\al}a_{k1}^{-1}, ~~~\quad \al >k  
\end{cases} \!\!,
\end{array}
$$
$\alpha,\beta=2,\ldots  n$, and extended as algebra map to all $\cO_q(P)$,
 is a well defined $\cO_q(P)$-comodule algebra map for any $k=1,\ldots n$.
\end{proposition}
\begin{proof} This is a direct check similar to Proposition
\ref{lemmaqpb1}. Recalling the  commutation relations of the
$p_{ij}$'s (cf. proof of Proposition \ref{lemmaqpb1}), and those  between
quantum minors in \cite{fi3}, we have:
i) the $a_{kj}$ among themselves have the same commutation relations as
the $p_{1j}$'s. ii)  $a_{k1}$ commutes with $D_{\al k}^{1\beta},
D_{1k}^{1\beta}, D_{k\al }^{1\beta}$. iii) The
$-qa_{k1}^{-1}\!D_{\al k}^{1\beta}$'s, satisfy the same Manin relations
  among themselves as the $p_{\al\beta}$'s; similarly for the 
$a_{k1}^{-1}\!D_{1k}^{1\beta}$'s and the
$a_{k1}^{-1}\!D_{k\al}^{1\beta}$'s. iv) The mixed commutation relations:
of $-qa_{k1}^{-1}\!D_{\al k}^{1\beta}$ 
 with $a_{k1}^{-1}\!D_{1k}^{1\beta}$ and with
 $a_{k1}^{-1}\!D_{k\al}^{1\beta}$, and of
 $a_{k1}^{-1}\!D_{1k}^{1\beta}$ with 
 $a_{k1}^{-1}\!D_{k\al}^{1\beta}$, also satisfy the same Manin
 relations as those of the corresponding $p_{\al\beta}$'s.

Then we are left to check the commutation relations
of $a_{k\gamma}$ with  $-q a_{k1}^{-1} \!D^{1\beta}_{\alpha k}$,
$a_{k1}^{-1} \!D^{1j}_{ki}$ and  $a_{k1}^{-1}\!D_{k\al}^{1\beta}$.
There are nine of these, depending on the combinations $k>\al, k=\al, k<\al$ with
$\gamma> \beta, \gamma=\beta, \gamma <\beta$. These indeed correspond 
to the commutation relations between  $p_{1\gamma}$ and $p_{\al\beta}$.   

We conclude that $j_k$ is a well defined algebra map because in
$\cO_q(\rSL_n)[a_{k1}^{-1}]$ we have
$j_k(p_{11})j_k(\det_q(p_{\al\beta}))=1$, consistently with the last of the
defining relations of the algebra $\cO_q(P)$:
$p_{11}\det_q(p_{\al\beta})=1$. This is obtained with the same
argument as in Lemma \ref{lemmaqpb1bis}.

Since $j_k$ is an algebra map it is an $\cO_q(P)$-comodule map provided
the comodule property $\de_1 \circ j_1 = ( j_1 \otimes \mathrm{id})
\circ  \Delta_P$ holds on the generators. It is straighforward to see
that this is indeed the case on $p_{1j}$. Let's compute the case
$p_{\al\beta}$ with $\alpha>k$ (the other
cases being similar):
$$ 
\begin{array}{rl}
\!\!\!\!\!\!(\de_k \circ j_k)(p_{\alpha\beta})\!\!\!&=\,
\de_k(a_{k1}^{-1})\, \de_k(D^{1\beta}_{k\alpha})=
\big(a_{k1}^{-1}\otimes \pi(a_{11}^{-1})\big)\sum_{r<s} D^{rs}_{k\alpha} \otimes 
\pi(D^{1\beta}_{rs}) \\[.7em]
&=\,\sum_{r<\gamma} a_{k1}^{-1} \!D^{r\gamma}_{1\alpha} \otimes \pi(a_{11}^{-1} \!D^{1\beta}_{r\gamma})=
\sum_{\gamma} a_{k1}^{-1}  \!D^{1\gamma}_{1\alpha}\otimes p_{\gamma\beta}\,\\[.7em]
&=\big(( j_1 \otimes \mathrm{id}) \circ
  \Delta_P\big)(p_{\al\beta})\,.\end{array}
$$\\[-2em]
\end{proof}

\begin{remark}Recalling
Remark \ref{cleft-rem} and Observation \ref{obsjtheta}, as corollary of the above
proposition we have
$\cF(U_k)\simeq \C_q[a_{i1}a_{k1}^{-1}]^{}_{i\in {\cal I}_k}\,\#
\,\cO_q(P)$, ${\cal I}_k:=\{i; 1\!\leq i\!\leq n, i\not= k\}$, where
it is easy to check that the smashed product is nontrivial
(i.e., different from the tensor product).
\end{remark}

\begin{theorem}\label{mainproj-thm} 
Let the notation be as in the previous section.
The assignment:
$$
U_I \mapsto \cF(U_I):=\cO_q(\rSL_n)S_{i_1}^{-1}\dots S_{i_s}^{-1},
\qquad I=\{i_1, \dots, i_s\}
$$
defines a quantum principal
bundle on 
the quantum ringed space \break
$(\rSL_n(\C)/P, \cF^{\coinv \, \cO_q(P)})$, with  structure sheaf $\cF^{\coinv \, \cO_q(P)}=\cO_{\rSL_n/P}$
given by projective localizations of the quantum homogeneous projective  space $\tilde{\cO}_q(\bP^{n-1})=\tilde{\cO}_q(\rSL_n/P)$.
\end{theorem}

\begin{proof}
After Proposition \ref{shf-def-prop2} we only need to prove
the locally cleft property. This is a direct consequence of
Proposition \ref{lemmaqpb2} and Remark \ref{cleft-rem}.
\end{proof}

\begin{remark}
Notice that our construction, and in particular
Theorem \ref{mainproj-thm}, holds also when we take $q \in \C$, that is, we
specialize the indeterminate $q$ to a complex value.

\end{remark}

\section{Quantum principal bundles from twists}\label{TW}
In this section we obtain new quantum principal bundles via
2-cocycle deformations. 
In particular we provide examples that are locally cleft from
examples that are locally trivial.

We here consider the ground ring to
be a field, hence specialize $q\in k$.
As in \cite{abps} we consider 2-cocycle (twist) deformations 
based on the ``structure group'' Hopf algebra $H$ and also on an ``external symmetry'' Hopf
algebra $K$, i.e. a Hopf algebra coacting on the quantum principal
bundle, the coaction being compatible with that of $H$
(in the commutative case $K$ is  associated with automorphisms of the
bundle, possibly nontrivial on the base).

\subsection{Deformations from twists of  $H$}
Let $\gamma: H\otimes H\to k$ be a 2-cocycle of the Hopf algebra
$H$, denote by  $\gamma^{-1}: H\otimes H\to k$ its convolution inverse
and by $H_\gamma$ the new Hopf algebra that has the same
costructures of $H$ and new product $\cdot_\gamma$ and antipode  obtained by twisting the
ones of $H$ via $\gamma$. Explicitly the product reads, for all $h,h'\in H$, $h\cdot_\gamma
h'=\gamma(h_{(1)}\otimes h'_{(1)})h_{(2)}h'_{(2)}
\gamma^{-1}(h_{(3)}\otimes h'_{(3)})$.
We also denote with $\Gamma$ the functor from the category of right $H$-comodule
algebras to that of right $H_\gamma$-comodule algebras: if $A$ is an
$H$-comodule algebra then $\Gamma(A)\equiv A_\gamma$ is the $\kk_q$-module $A$
with product $a \bullet_\gamma a':=a_{(0)}a'_{(0)}\gamma^{-1}(a_{(1)}\otimes
a'_{(1)})$. Since $H$ and $H_\gamma$ have the same costructures,
$A_\gamma$ is a right $H_\gamma$-comodule algebra using the same
comodule structure map as for $A$. The functor $\Gamma$ is the identity on morphisms.

\begin{theorem}\label{TWIST}
Let $\gamma$ be a 2-cocycle of the Hopf algebra
$H$ and $\Gamma$ the corresponding functor of comodule
algebras. The sheaf $\cF$ is an $H$-principal bundle (quantum
principal bundle) over the ringed space
$(M,\cO_M)$ if and only if $\Gamma\circ \cF$ is an $H_\gamma$-principal bundle over
$(M,\cO_M)$.
\end{theorem}
\begin{proof}
If $\cal F$ is a sheaf of $H$-comodule algebras over 
$M$ then  $\Gamma\circ \cF$ is easily seen to be a sheaf of 
$H_\gamma$-comodule algebras over
$M$ (locality and the gluing property  immediately follow
recalling that $\Gamma$ is the identity on objects). Vice versa, since
the convolution inverse
$\gamma^{-1}$ is a 2-cocycle of $H_\gamma$, and
$(H_\gamma)_{\gamma^{-1}}=H$, if $\Gamma\circ \cF$ is a sheaf of
$H_\gamma$-comodule algebras then  $\cF=\Gamma^{-1}\circ (\Gamma\circ
\cF)$ is a sheaf of $H$-comodule algebras.

Let $\{U_i\}$ be a covering of $M$ with $\cF(U_i)^{\coinv
  H}=\cO_M(U_i)$ and such that $\cF$ is locally cleft. Since
$H_\gamma$ and $H$ have the same coproduct we have $\cF(U_i)^{\coinv
  H_\gamma}=\cF(U_i)^{\coinv H}=\cO_M(U_i)$ as algebras. Finally, $\cF(U_i)^{\coinv H}\subset
\cF(U_i)$ is a cleft extension if and only if  $\cF(U_i)^{\coinv
  H}\subset \cF(U_i)_\gamma$ is a cleft extension,
cf. \cite[Theorem 5.2]{MS} or \cite[Corollary
3.7]{abps}.
\end{proof}

\begin{remark}\label{ltlc}
We further observe that if the $H$-principal bundle $\cF$ is locally
trivial with respect to a covering $\{U_i\}$, i.e., the cleft extensions $\cF(U_i)^{\coinv
  H}\subset \cF(U_i)$ are trivial extensions, so that $\cF(U_i)\simeq \cF(U_i)^{\coinv
  H}\sharp H$ (cf.~Observation \ref{obsjtheta}),  then this is no more the
case for the twisted $H_\gamma$-principal bundle $\Gamma\circ \cF$
because the extensions  $\cF(U_i)^{\coinv
  H_\gamma}\subset \cF(U_i)_\gamma$  are cleft but nontrivial. Indeed, as
follows from \cite[Theorem 5.2]{MS}, 
$\cF(U_i)_\gamma\simeq\cF(U_i)^{\coinv H}\sharp_{\gamma^{-1}} H_\gamma$,
where $\sharp_{\gamma^{-1}}$  denotes the crossed product
given by the 2-cocyle $\gamma^{-1}$ of $H_\gamma$.
\end{remark}

\subsection{Deformations from twists of
  $K$}
Let now $K$ be another Hopf algebra and $\cF$ be a sheaf over the
ringed space $(M,\cO_M)$ of $(K,H)$-bicomodule algebras, i.e. right
$H$-comodule algebras and left $K$-comodule algebras with left
and right
 coactions commuting: $(\rho\otimes id)\circ \delta =(id\otimes \delta)\circ \rho$.
Since $\kk$ is a field, $K$ is free as a $\kk$-module and $\cF^{\coinv H}: U\to \cF(U)^{\coinv H}$ is
a subsheaf of $K$-comodule algebras (because $\cF(U)^{\coinv H}$
are $K$-subcomodule algebras, cf. \cite[Proposition 3.12]{abps}).

A twist
$\sigma$ of $K$ gives the functor $\Sigma$  from
left $K$-comodule algebras $A$ to left $K_\sigma$-comodule algebras
$\Sigma(A)\equiv {}_\sigma A$, where the new product is given by 
$a\,\mbox{${}_\sigma\;\!\!\bullet_{\,}$}
a'=\sigma(a_{(-1)}\otimes a'_{(-1)}) a_{(0)}a'_{(0)}$ (the comodule
structure maps of $A$ and ${}_\sigma A$ being the same). The functor
$\Sigma$ is the identity on morphisms.
As in Theorem \ref{TWIST}, composition of this functor with the sheaf
$\cF$ of $(K,H)$-bicomodule algebras gives the sheaf $\Sigma\circ \cF$
of $(K_\sigma,H)$-bicomodule algebras. 
\begin{theorem}\label{SIGMAsheaf} 
Let the sheaf  $\cF$ of $(K,H)$-bicomodule algebras 
over the ringed space $(M,\cF^{\coinv H})$ 
be an $H$-principal bundle. 
If the $H$-comodule $(H,\Delta)$ has a compatible $K$-comodule
structure, so that it is a $(K,H)$-bicomodule
and the cleaving maps $j_i:H\rightarrow \cF(U_i)$,
relative to a covering $\{U_i\}$ of $M$, are $(K,H)$-bicomodule maps,
then the sheaf $\,\Sigma \circ \cF$ of  $(K_\sigma,H)$-bicomodule
algebras over the ringed space $(M,\Sigma\circ \cF^{\coinv H})$ is an
$H$-principal bundle. 
\end{theorem}
\begin{proof} Since the sheaf $\cF^{\coinv H}$  of
$K$-comodule algebras 
is a subsheaf of the sheaf  $\cF$ of $K$-comodule algebras the sheaf
$\Sigma\circ \cF^{\coinv H}$ of $K_\sigma$-comodule algebras is well
defined. Since the $\Sigma$ functor is the identity on objects
$\Sigma\circ \cF^{\coinv H}=(\Sigma\circ \cF)^{\coinv H}$ as
$K_\sigma$-comodule algebras.

We are left to show that the sheaf $\Sigma\circ \cF$ is locally
cleft. From Theorem \ref{thmDT}, for each open $U_i$ we have the local trivialization
\begin{equation}\label{thetaF}
\vartheta_i: \cF(U_i)^{\coinv H}\otimes H\rightarrow
\cF(U_i)~,~~b\otimes h\mapsto \vartheta_i(b\otimes h)=bj_i(h) 
\end{equation} that is an isomorphism of left $ \cF(U_i)^{\coinv H}$-modules and right
$H$-comodules. Since $j_i$ is also a left $K$-module map and
$\cF(U_i)$ is a $K$-comodule algebra we easily have that $\vartheta_i$
is also a left $K$-comodule map.

Recall that a twist $\sigma$ defines a monoidal functor $(\Sigma, \varphi^\ell)$ from the
category of left $K$-comodules
$({}^{K\!}{\cal M},\otimes)$ to that of left $K_\sigma$-comodules
$({}^{K_\sigma\!}{\cal M},{}^\sigma\otimes)$, where ${}^\sigma\otimes$
and $\otimes$ coincide as tensor products of $\kk$-modules.
The functor $\Sigma: {}^K{\cal M}\to {}^{K_\sigma\!}{\cal M}$, $V\mapsto \Sigma(V)\equiv
{}_\sigma V$ is the identity on objects and morphisms because as
coalgebras $K= K_\sigma$, while the natural transformation
$\varphi^\ell$ between the tensor product functors $\otimes$ and
${}^\sigma\otimes$ is given by the $_\sigma K$-comodule isomorphisms
$\varphi^\ell_{VW}:\Sigma(V\otimes W)\to
\Sigma(V){\,}^\sigma\!\!\otimes\, \Sigma(W)$, $v\otimes w\mapsto
\varphi^\ell_{MN}(v\otimes w)=\sigma(v_{(-1)}\otimes w_{(-1)})\, v_{(0)}\otimes w_{(0)}$,
where $\rho(v)=v_{(-1)}\otimes v_{(0)}$, $\rho(w)=w_{(-1)}\otimes
w_{(0)}$ are the left $K$-coactions of $V$ and $W$.

Furthermore, $(\Sigma,
\varphi^\ell)$  is a monoidal functor from the category  of $(K,H)$-bicomodules 
$({}^{K\!}{\cal M}^H,\otimes)$ to that of $(K_\sigma,H)$-bicomodules 
$({}^{K_\sigma\!}{\cal M}^H,{}^\sigma\otimes)$,
(cf.~for example \cite[\S 2.2]{abps}).

Applying the functor $\Sigma$ to the $\cF(U_i)^{\coinv H}$-module and 
 $(K,H)$-bicomodule isomorphism $\vartheta_i$ 
we obtain the  isomorphism of  left ${}_\sigma\cF(U_i)^{\coinv H}$-modules and 
$(K_\sigma,H)$-bicomodules
$$\Sigma(\vartheta_i): {}_\sigma\big(\cF(U_i)^{\coinv H}\otimes
H\big)\rightarrow {}_\sigma\cF(U_i)~,$$
where ${}_\sigma\big(\cF(U_i)^{\coinv H}\otimes
H\big):= \Sigma(\cF(U_i)^{\coinv H}\otimes
H)$ and ${}_\sigma\cF(U_i):= \Sigma(\cF(U_i))$.
Using the $(K_\sigma,H)$-bicomodule isomorphism (we suppress the pedices of
$\varphi^\ell$ for simplicity)
$$\varphi^\ell:  {}_\sigma\big(\cF(U_i)^{\coinv H}\otimes
H\big)\rightarrow  {}_\sigma\cF(U_i)^{\coinv H}{\:}^\sigma\!\!\otimes\,
{}_\sigma H~,$$  where $ {}_\sigma H:=\Sigma(H)$ is just the
$(K,H)$-bicomodule $H$ now seen as a
$(K_\sigma,H)$-bicomodule, we obtain the left $ {}_\sigma\cF(U_i)^{\coinv H}$-module and 
$(K_\sigma, H)$-bicomodule isomorphism
$$\Sigma(\vartheta_i)\circ {\varphi^\ell\,}^{-1}\!:  {}_\sigma\cF(U_i)^{\coinv H}\;{}^\sigma\!\!\otimes\,
{}_\sigma H\,\rightarrow \,{}_\sigma\cF(U_i)~.$$ Forgetting the
$K_\sigma$-comodule structure and recalling that as $H$-comodules
$_\sigma H=H$, and that as tensor products of $H$-comodules we have ${}^\sigma\otimes=\otimes$,
this isomorphism becomes an ${}_\sigma\cF(U_i)^{\coinv H}$-module and 
$H$-comodule isomorphism 
$ {}_\sigma\cF(U_i)^{\coinv H}\otimes
H\rightarrow {}_\sigma\cF(U_i)$,
proving that the  extension ${}_\sigma\cF(U_i)^{\coinv H}\subset  {}_\sigma\cF(U_i)$ 
 is cleft. This holds for each open $U_i$, thus  $\Sigma\circ \cF$ is locally cleft.
\end{proof}

\subsection{Examples}
We twist the quantum principal bundle $\cF$ on the quantum
ringed space $(\rSL_n(\C)/P, \cF^{\coinv \, \cO_q(P)})$ of Theorem
\ref{mainproj-thm} and obtain three new quantum principal bundles: $\Gamma\circ \cF$, $\Sigma\circ  \cF$ and $\Gamma\circ\Sigma\circ
\cF$; the first on the locally ringed space associated with the homogeneous
ring of quantum projective space $\tilde\cO_q(\bP^{n-1})$, the
other two on its
multiparametric deformation $\tilde\cO_{q,\gamma}(\bP^{n-1})$. 
\\

\noindent{\it Deformations from twists of $H=\cO_q(P)$.}\\ 
The
$(n-1)$-dimensional torus $\mathbb{T}^{n-1}$ is a subgroup of $SL_n(\C)$
and correspondingly we have that the Hopf algebra $\cO(\mathbb{T}^{n-1})$ (the
group Hopf algebra over $\C$ of the free abelian group generated by $n-1$
elements) is a quotient of $\cO_q(SL_n) $. It is useful to present $\cO(\mathbb{T}^{n-1})$ as the algebra over $\C$ generated
by the $n$ elements $t_i$, $i=1,\ldots n$ and their inverses $t^{-1}_i$
modulo the ideal generated by  the relation $t_1t_2\ldots t_n=1$. The
Hopf algebra structure is fixed by requiring $t_i$ to be group like.
The Hopf algebra projection $\cO_q(SL_n) \stackrel{pr}{\longrightarrow}
\cO(\mathbb{T}^{n-1})$ on the generators is given by $$a_{ij}\mapsto \delta_{ij}t_i~.$$
We consider the exponential  2-cocycle $\gamma$  on $\mathcal{O} (\mathbb{T}^{n-1})$
defined on the generators $t_i$ by
\be\label{cocycle-exp}
\co{t_j}{t_k}= \gamma_{jk} \mbox{~~ with~ } \gamma_{jk}=\exp\big(i \pi \theta_{jk}\big)  ~;\quad \theta_{jk}=- \theta_{kj} \in \mathbb{R}
\ee
and  extended to the whole algebra via
\begin{equation}\label{twistabc}
\co{ab}{c}=\co{a}{\one{c}}\co{b}{\two{c}}~~,~\co{a}{bc}=\co{\one{a}}{c}\co{\two{a}}{b}
\end{equation} 
for all $a,b,c, \in \mathcal{O} (\mathbb{T}^n)$. 
This 2-cocycle $\gamma$ is pulled back
along the projection $\cO_q(SL_n) \stackrel{pr}{\longrightarrow}
\cO(\mathbb{T}^{n-1})$ to a 2-cocycle $\gamma\circ (pr\otimes pr)$ on
$\cO(\rSL_n)$ (see e.g. \cite[Lemma
4.1]{abps}). Explicitly, denoting with abuse of notation by $\gamma$
the pulled back 2-cocycle, we have that
\begin{equation}\label{TWonSL}
\gamma: \cO_q(SL_n)\otimes \cO_q(SL_n)\rightarrow \C
\end{equation}
is defined by 
$\gamma(a_{ij}\otimes a_{kl})=\delta_{ij}\delta_{kl}\gamma_{il}$, and 
\eqref{twistabc} for all $a,b,c \in \mathcal{O}_q (SL_n)$.
Twist deformation via this 2-cocycle of the quantum group $\mathcal{O}_q (SL_n)$ gives
the multiparametric special linear quantum group studied e.g. in \cite{Res}.

The torus Hopf algebra $\cO(\mathbb{T}^{n-1})$ is also a quotient of the
parabolic quantum group $\cO_q(P)$ defined in
\eqref{qparSL}. Correspondingly the 2-cocycle $\gamma$ on  $\cO(\mathbb{T}^{n-1})$.
is pulled back to a 2-cocycle,  still denoted $\gamma$, on $\cO_q(P)$
providing its multiparametric deformation $\cO_{q,\gamma}(P)$.
\\

We now apply Theorem \ref{TWIST} to the $\cO_q(P)$-principal bundle $\cF$ on the quantum
ringed space $(\rSL_n(\C)/P, \cF^{\coinv \, \cO_q(P)})$ of Theorem
\ref{mainproj-thm} and obtain the $\cO_{q,\gamma}(P)$-principal bundle
$\Gamma\circ \cF$ on  $(\rSL_n(\C)/P, \cF^{\coinv \,  \cO_q(P)})$. 
Furthermore, Remark \ref{ltlc} implies that  
 while the  $\cO_q(P)$-principal bundle $\cF$ is locally
trivial on the cover $\{U_i\}$ of $\bP^{n-1}(\C)=\rSL_n(\C)/P$, the $\cO_{q,\gamma}(P)$-principal bundle
$\Gamma\circ \cF$ is only locally cleft.
\\

\noindent{\it Deformations from twists of $K=\cO(\mathbb{T}^{n-1})$.}\\
We next study twists based on the external Hopf algebra $K=\cO(\mathbb{T}^{n-1})$.
The $\cO_q(P)$-principal bundle $\cF$ on  
$(\rSL_n(\C)/P, \cF^{\coinv \, \cO_q(P)})$ of Theorem \ref{mainproj-thm} is indeed a sheaf of
$(\cO(\mathbb{T}^{n-1}),\cO_q(P))$-bicomodule algebras: 
The left $K=\cO(\mathbb{T}^{n-1})$-coaction on the $\cO_q(P)$-comodule
algebra  $\cO_q(\rSL_n(\C))$ is given by $$\rho(a)=(pr \otimes
id)\Delta^{\phantom{{J_J}_J}}_{\cO_q(\rSL_n(\C))}(a)$$ for all $a\in \cO_q(\rSL_n(\C))$,
and is uniquely extended as algebra map to the sheaf 
$U_I\mapsto \cF(U_I)=\cO_q(\rSL_n(\C))S^{-1}_{i_1}\ldots
S^{-1}_{i_s}$, $I=\{i_1\ldots i_s\}$
of $\cO_q(P)$-comodule algebras on $\rSL_n(\C)/P$, where $\{U_I\}$ is the topology 
on $\rSL_n(\C)/P$ generated by the cover $\{U_i\}$.

Furthermore, the cleaving maps $j_i: \cO_q(P)\to \cF(U_i)=\cO_q(\rSL_n(\C))S^{-1}_i$
become $(\cO(\mathbb{T}^{n-1}),\cO_q(P))$-comodule maps by defining on the
$\cO_q(P)$-comodule $(\cO_q(P),\Delta)$ the compatible left
$\cO(\mathbb{T}^{n-1})$-comodule structure given by 
$\rho(a)=(p \otimes id)\Delta(a)$, where $p$ is the projection
$\cO_q(P)\stackrel{p}{\longrightarrow} \cO(\mathbb{T}^{n-1})$.
We can then consider the 2-cocycle \eqref{cocycle-exp} for
$K=\cO(\mathbb{T}^{n-1})$ and  apply Theorem \ref{SIGMAsheaf} thus concluding that the sheaf
$\Sigma\circ \cF$ is an $\cO_q(P)$-principal bundle over the ringed space
$(\bP^{n-1}(\C), \Sigma\circ \cF^{\coinv\, \cO_q(P)})$. In Remark \ref{64}
we further show it is not locally trivial on the cover $\{U_i\}$.
\\

\noindent{\it Deformations from both twists of   $H=\cO_q(P)$ and $K=\cO(\mathbb{T}^{n-1})$.}\\
Finally, we can consider the 
$\cO_q(P)$-principal bundle $\Sigma\circ \cF$  over the ringed space
$(M, \Sigma\circ \cF^{\coinv\, \cO_q(P)})$, and use the 2-cocycle of
$\cO_q(P)$, obtained via pullback of the 2-cocycle \eqref{cocycle-exp} of  $\cO(\mathbb{T}^{n-1})$,
in order to construct, according to Theorem \ref{TWIST}, the
$\cO_{q,\gamma}(P)$-principal bundle $\Gamma\circ \Sigma\circ \cF$  over the ringed space
$
(\bP^{n-1}(\C), \Sigma\circ \cF^{\coinv\,
  \cO_{q}(P)})$.

 Equivalently the $\cO_{q,\gamma}(P)$-principal bundle
$\Gamma\circ \Sigma\circ \cF$  is over $(\bP^{n-1}(\C),
\Sigma\circ (\Gamma\circ \cF)^{\coinv\, \cO_{q,\gamma}(P)})$,
 since $
(\bP^{n-1}(\C), \Sigma\circ \cF^{\coinv\,
  \cO_{q}(P)})=(\bP^{n-1}(\C),
\Sigma\circ (\Gamma\circ \cF)^{\coinv\, \cO_{q,\gamma}(P)})$, as follows
from $\cO_{q,\gamma}(P)$ and $\cO_{q}(P)$ having the same
coproduct. 

This $\cO_{q,\gamma}(P)$-principal bundle
$\Gamma\circ \Sigma\circ \cF$  is locally trivial with cleaving maps $(\Sigma\circ\Gamma)
(j_i): \cO_{q,\gamma}(P)\rightarrow  (\Sigma\circ \Gamma\circ \cF)
(U_i)$ that are algebra maps since  
$j_i: \cO_{q}(P)\rightarrow   \cF
(U_i)$ in Proposition \ref{lemmaqpb2} are $(\cO(\mathbb{T}^{n-1}),\cO_q(P))$-bicomodule
algebra maps.
\\

We now show that this $\cO_{q,\gamma}(P)$-principal bundle
$\Gamma\circ \Sigma\circ \cF$  is an example of the construction of Theorem \ref{main}. 
This is so because the (graded) algebras $\cO_{q}(\rSL_n)$,
$\cO_{q}(P)$, $\cO_q(\rSL_n/P)$ and their localizations  are left and right
(graded) $\cO(\mathbb{T}^{n-1})$-comodule algebras. 

We first observe that the total space (global sections) of $\Gamma\circ
\Sigma\circ \cF$  is the multiparametric
quantum group
\begin{equation}\label{OSLqg} (\Gamma\circ \Sigma\circ
\cF)(\rSL_n(\C))=\cO_{q,\gamma}(\rSL_n)~,
\end{equation}
with $\cO_{q,\gamma}(P)$ that is a quantum subgroup. Indeed we can pull back
the twist \eqref{cocycle-exp} on $K=
\cO(\mathbb{T}^{n-1})$ to the twist \eqref{TWonSL} on $\cO_q(\rSL_n)$.
Then  $(\Gamma\circ\Sigma)(\cO_q(\rSL_n))$ is the twist of
$\cO_q(\rSL_n)$ as a left $\cO_q(\rSL_n)$-comodule algebra and with the same
twist   \eqref{TWonSL} as a right  $\cO_q(\rSL_n)$-comodule algebra, hence
it is the twisting of $\cO_q(\rSL_n)$ as a Hopf algebra, giving the
Hopf algebra $\cO_{q,\gamma}(\rSL_n)$. Similarly we have 
\begin{equation}\label{OPqg}
(\Gamma\circ \Sigma)(\cO_q(P))=\cO_{q,\gamma}(P)~.
\end{equation} 
In order to show that $\cO_{q,\gamma}(P) $ is a quantum subgroup of
$\cO_{q,\gamma}(\rSL_n)$ recall
that the deformation \eqref{OSLqg} is induced from a left and right action of the
Hopf algebra $\cO(\mathbb{T}^{n-1})$ and notice that the ideal $I_q(P)=(a_{\al 1})\subset \cO_q(\rSL_n)$ is a left and right 
$\cO(\mathbb{T}^{n-1})$-subcomodule algebra. Its twist
deformation  $I_{q,\gamma}(P):=(\Sigma\circ \Gamma) (I_q(P))$
is an ideal in $\cO_{q,\gamma}(\rSL_n)$. It is furthermore a
Hopf ideal since so was $I_q(P)$ in $\cO_q(\rSL_n)$, and because twisting does not
affect the costructures and twisting via the exponential 2-cocycle
\eqref{cocycle-exp} does not affect the antipode as a linear map.  We can then consider
 the quotient Hopf algebra
 $\cO_{q,\gamma}(\rSL_n)/I_{q,\gamma}(P)$, this is easily seen to be the
 multiparametric quantum group in \eqref{OPqg}. 

We next twist $\tilde{\cO}_{q}(\bP^{n-1})=\tilde{\cO}_{q}(\rSL_n/P)$ seen as left 
$K=\cO(\mathbb{T}^{n-1})$-comodule algebra (and a trivial right
$\cO(\mathbb{T}^{n-1})$-comodule algebra). The twist is grade preserving and
therefore  $\tilde{\cO}_{q,\gamma}(\bP^{n-1}):=(\Sigma\circ \Gamma)
(\tilde{\cO}_{q}(\bP^{n-1}))$ is a graded algebra. It is generated by the quantum section $d=a_{11}\in
\tilde{\cO}_{q,\gamma}(\rSL_n)$ and the corresponding $d_i=a_{i1}$ obtained
from the coproduct (that equals that of  $\cO_{q}(\rSL_n)$).
Indeed monomials in
$d_i$, respectively contructed with the product of $\tilde{\cO}_{q}(\bP^{n-1})$ and of
$\tilde{\cO}_{q,\gamma}(\bP^{n-1})$, differ by a phase and hence span
the same $\mathbb{C}$-module 
$\tilde{\cO}_{q,\gamma}(\bP^{n-1})$. Explicitly
 $\tilde{\cO}_{q,\gamma}(\bP^{n-1})$ is the subalgebra generated by the elements $x_i:=d_i=a_{i1}\in
\cO_{q,\gamma}(\rSL_n)$, i.e. it is 
the  multiparametric quantum homogeneous projective space 
$$
\tilde{\cO}_{q,\gamma}(\bP^{n-1})=\C_q\langle x_1, \dots x_{n}\rangle/(x_ix_j-q^{-1}\gamma_{ij\,}^2x_jx_i, i < j)~. 
$$

We now observe that $\cO_{q}(\rSL_n)S_i^{-1}$ is canonically an
$\cO(\mathbb{T}^{n-1})$-bicomodule algebra. We twist it to $(\Sigma\circ
\Gamma)(\cO_{q}(\rSL_n)S_i^{-1})$ and denote by
${}_\gamma\bullet_\gamma$ the corresponding product (notice that
${}_\gamma\bullet_\gamma$ restricted to the sub $\cO(\mathbb{T}^{n-1})$-bicomodule $\cO_q(\rSL_n)$ is
the Hopf algebra twist of the product of $\cO_q(\rSL_n)$). Due to $\gamma(t_i^{-1}\otimes
t_i) =1=\gamma(t_i\otimes
t_i^{-1})$ (cf. \eqref{cocycle-exp} and \eqref{twistabc}), we have
$d_i^{-1}{}_\gamma\!\bullet_{\gamma\,} d_i=d_i^{-1}d_i$ and  $d_{i\:}{}_{\gamma\;}\!\!\bullet_{\gamma\,}
d_i^{-1}=d_id_i^{-1}$. This shows that the inverse $d_i^{-1} $ of $d_i$ in $\cO_q(\rSL_n)$ is
also the inverse  in $\cO_{q,\gamma}(\rSL_n)$. 

Then the identity $(a{\,}_\gamma\!\bullet_{\gamma\,}  d_i^{-1}){}_\gamma\!\bullet_{\gamma\,} d_i=
a{\,}_\gamma\!\bullet_{\gamma\,} (d_i^{-1}{}_\gamma\!\bullet_{\gamma\,} d_i)=a{\,}_\gamma\!\bullet_{\gamma\,} (d_i^{-1}d_i)=a$, where
$a\in \cO_q(\rSL_n)$, and more in general $a\in  (\cO_q(\rSL_n)
S_{i_1}^{-1}...  \widehat{S_{i_s}^{-1}}... S_{i_s}^{-1})_\gamma$,
shows that the  twist of the
localizations of $\cO_{q}(\rSL_n)$, are just the  localizations of the twisted quantum group $\cO_{q,\gamma}(\rSL_n)$, i.e.,
$$(\Sigma\circ \Gamma\circ \cF) (U_I):=(\Sigma\circ\Gamma)(\cO_q(\rSL_n) S_{i_1}^{-1}\ldots S_{i_s}^{-1})=
\cO_{q,\gamma}(\rSL_n) S_{i_1}^{-1}\ldots S_{i_s}^{-1}~,$$
$I=\{i_1,\ldots i_s\}$. 
This shows that the Ore conditions are satisfied for the  localizations of $
\cO_{q,\gamma}(\rSL_n)$ and that the corresponding sheaf constructed as in Theorem \ref{main} is $\Sigma\circ
\Gamma\circ \cF$ .
We summarize this result in the following theorem.
\begin{theorem}\label{mainproj-thmTW} 
The assignment:
$$
U_I \mapsto (\Sigma\circ \Gamma \circ \cF)(U_I)=\cO_{q,\gamma}(\rSL_n)S_{i_1}^{-1}\dots S_{i_s}^{-1},
\qquad I=\{i_1, \dots, i_s\}
$$
defines a quantum principal
bundle on 
the quantum ringed space \break
$(\rSL_n(\C)/P, (\Sigma\circ \Gamma \circ \cF)^{\coinv \,
  \cO_{q,\gamma}(P)})$, with  structure sheaf  $(\Sigma\circ \Gamma \circ \cF)^{\coinv \, \cO_{q,\gamma}(P)}$
given by projective localizations of the multiparametric quantum homogeneous projective  space $\tilde{\cO}_{q,\gamma}(\bP^{n-1})=\tilde{\cO}_{q,\gamma}(\rSL_n/P)$.
\end{theorem}

\begin{remark}
An immediate application of this result is that the   $\cO_{q,\gamma}(P)$-principal bundle
$\Gamma\circ \Sigma\circ \cF$ is locally trivial with cleaving maps $(\Sigma\circ\Gamma)
(j_i): \cO_{q,\gamma}(P)\rightarrow  (\Sigma\circ \Gamma\circ \cF)
(U_i)$ that are algebra maps (recall Remark \ref{cleft-rem}). Indeed   
$j_i: \cO_{q}(P)\rightarrow   \cF
(U_i)$ in Proposition \ref{lemmaqpb2} are $(\cO(\mathbb{T}^{n-1}),\cO_q(P))$-bicomodule
algebra maps, and the result follows applying the functor $\Gamma\circ
\Sigma$ and recalling \eqref{OPqg}. 
\end{remark}

\begin{remark}\label{64}
Since the left and right coactions commute we have $\Sigma\circ \Gamma=\Gamma\circ \Sigma$ (cf. \cite[Proposition
2.27]{abps})  and hence $\Sigma\circ \cF=\Gamma^{-1}\circ
(\Sigma\circ\Gamma\circ 
\cF)$. Applying  Remark \ref{ltlc} to the locally trivial bundle $\Sigma\circ\Gamma\circ
\cF$ (and considering the functor $\Gamma^{-1}$ instead of $\Gamma$)
we conlcude that
the extensions $(\Sigma\circ \cF)(U_i)^{\,\coinv\,
  {\cO_{q}(P)}}\subset (\Sigma\circ \cF)(U_i)$
are cleft and nontrivial.
So that the ${\cO_{q}(P)}$-principal bundle $\Sigma\circ \cF$  locally
is cleft and nontrivial.
\end{remark}

\end{document}